\documentclass[abstract=on]{scrartcl}
\setkomafont{sectioning}{\normalfont\normalcolor\bfseries}
\usepackage{amsmath}
\usepackage{amssymb}
\usepackage{amsthm}
\usepackage{mathtools}
\usepackage[T1]{fontenc}
\usepackage[utf8]{inputenc}
\usepackage{setspace}
\usepackage{graphicx}
\usepackage{color}
\usepackage{cite}
\usepackage{lscape}
\usepackage{pdfpages}
\usepackage{caption}
\usepackage{subfig}
\usepackage{bbm} 
\usepackage{enumitem} 
\usepackage[usenames,dvipsnames]{xcolor}
\usepackage[color=SeaGreen,textsize=small]{todonotes} 
\usepackage{comment}
\usepackage{mathrsfs}
\usepackage{algorithm}
\usepackage{algpseudocode}
\usepackage{multirow}
\usepackage{longtable}
\usepackage{dcolumn}
\usepackage{booktabs}
\usepackage{mathdots}
\usepackage{hhline}
\usepackage[space]{grffile}
\usepackage{verbatim}
\usepackage{url}

\usetikzlibrary{patterns}
\usetikzlibrary{decorations.pathreplacing}
\usepackage[hyperfootnotes=false,hidelinks,bookmarksnumbered,bookmarksopen]{hyperref}
\hypersetup{colorlinks=false}
\usepackage{nameref}

\setenumerate{label=(\roman*)}

\makeatletter
\newcommand{\labitem}[2]{%
\def\@itemlabel{\text{#1}\textnormal{)}}
\item
\def\@currentlabel{#1}\label{#2}}
\makeatother

\makeatletter
\newcommand{\labitems}[2]{%
\def\@itemlabel{\text{#1}}
\text{\!\!#1\!\!}
\def\@currentlabel{#1}\label{#2}}
\makeatother

\makeatletter
\newcommand{\labitemc}[2]{%
\def\@itemlabel{\textbf{Case} \text{#1})}
\item
\def\@currentlabel{#1)}\label{#2}}
\makeatother

\usepackage{geometry}
\newtheoremstyle{mythm}
   {10pt}                   
   {10pt}                   
   {}          		    
   {}                      
   {\normalfont\bfseries}  
   {}                      
   { }
   {\textbf{\thmname{#1} \thmnumber{#2} \thmnote{(#3)}}} 

\newtheoremstyle{mythm2}
   {10pt}                   
   {10pt}                   
   {\itshape}          		    
   {}                      
   {\normalfont\bfseries}  
   {}                      
   { }
   {\textbf{\thmname{#1} \thmnumber{#2} \thmnote{(#3)}}} 

\newtheoremstyle{myex}
  {10pt}                   
   {10pt}                   
   {}          		    
   {}                      
   {\normalfont\bfseries}  
   {}                      
   { }
   {\textbf{\thmname{#1} \thmnumber{#2} \thmnote{(#3)}}} 

\theoremstyle{myex}
\newtheorem*{XxmpX}{Example} 
\newenvironment{ex}    
  {%
   \pushQED{\qed}\begin{XxmpX}}
  {\popQED\end{XxmpX}}
   
\theoremstyle{mythm2}
\newtheorem{thm}{Theorem}[section]
\newtheorem{lem}[thm]{Lemma}
\theoremstyle{mythm}

\newtheorem{Rem}[thm]{Remark} 
  {%
   \pushQED{\qed}\begin{Rem}}
  {\popQED\end{Rem}}

\DeclareMathOperator{\re}{Re}

\DeclareMathOperator{\argmax}{argmax}

\DeclareMathOperator{\diag}{diag}

\newcommand{\mods}{\,\bmod\,}
\newcommand{\RR}{\mathbb{R}}

\newcommand{\CC}{\mathbb{C}}
\newcommand{\NN}{\mathbb{N}}

\newcommand{\x}{\mathbf{x}}
\newcommand{\cxx}{\mathbf{x}^{\widehat{\mathrm{II}}}}
\newcommand{\cx}{x^{\widehat{\mathrm{II}}}}
\newcommand{\jj}{^{(j)}}
\newcommand{\jjj}{^{(j+1)}}
\newcommand{\aaa}{\mathbf{a}}
\newcommand{\uu}{\mathbf{u}}
\newcommand{\vv}{\mathbf{v}}
\newcommand{\y}{\mathbf{y}}
\newcommand{\z}{\mathbf{z}}

\newcommand{\W}{\mathbf{W}}

\newcommand{\bfeta}{\boldsymbol{\eta}}
\newcommand{\eps}{\varepsilon}
\newcommand{\J}{\mathbf{J}}
\newcommand{\F}{\mathbf{F}}
\newcommand{\C}{\mathbf{C}_N^{\mathrm{II}}}

\newcommand{\BIGOP}[1]{\mathop{\mathchoice%
{\raise-0.22em\hbox{\huge $#1$}}%
{\raise-0.05em\hbox{\Large $#1$}}{\hbox{\large $#1$}}{#1}}}

\begin{document}	
\title{Sparse Fast DCT for Vectors with One-block Support}
\author{Sina Bittens\thanks{University of Göttingen, Institute for Numerical and Applied Mathematics, Lotzestr. 16-18, 37083 Göttingen, 
Germany. Email: sina.bittens@mathematik.uni-goettingen.de} \and Gerlind Plonka\thanks{University of Göttingen, Institute for Numerical and Applied Mathematics, Lotzestr. 16-18, 37083 Göttingen, 
Germany. Email: plonka@math.uni-goettingen.de} }
\date{\today}
\maketitle
\begin{abstract}
 In this paper we present a new fast and deterministic algorithm for the inverse discrete cosine transform of type II that reconstructs the input vector $\mathbf{x}\in\RR^{N}$, $N=2^{J-1}$, with short support of length $m$ from its discrete cosine transform  $\mathbf{x}^{\widehat{\mathrm{II}}}=\C\x$.  
 The resulting algorithm has a runtime of $\mathcal{O}\left(m\log m\log \frac{2N}{m}\right)$ and requires $\mathcal{O}\left(m\log \frac{2N}{m}\right)$ samples of $\mathbf{x}^{\widehat{\mathrm{II}}}$.
 
 In order to derive this algorithm we also develop a new fast and deterministic inverse FFT algorithm that constructs the input vector $\y\in\RR^{2N}$ with reflected block support of block length $m$ from $\widehat{\y}$ with the same runtime and sampling complexities as our DCT algorithm.
\end{abstract}
\textbf{Keywords.}\quad discrete cosine transform, deterministic sparse fast DCT, sublinear sparse DCT, discrete Fourier transform, deterministic sparse FFT, sublinear sparse FFT
\newline
\textbf{AMS Subject Classification.}\quad 65T50, 42A38, 65Y20
\section{Introduction}

In recent years there has been an increased effort to develop deterministic sparse FFT algorithms that exploit a priori knowledge of the resulting vector to reduce the overall complexity. For example, if ${\mathbf x} \in {\mathbb C}^{N}$ is known to  possess $m$ significant entries,  or a short support of length $m$, the derived deterministic sparse FFT has only a sublinear arithmetical complexity.

There exist several approaches to obtain sublinear-time sparse FFT algorithms, mostly with different requirements on the supposed structure of the resulting vector. If the support in Fourier domain consists of $n$ blocks of length at most $m$, there is a deterministic sparse FFT algorithm with runtime $\mathcal{O}\left(\frac{mn^2\log m\log^4N}{\log^2n}\right)$, see \cite{bit_zha_iw}. Apart from other deterministic methods, see, e.g., \cite{akavia2014,bittens2016,christlieb2016multiscale,iwen,iwen_improved,plonka_smallsupp,plonka_nonneg,plonka_sparse,iwen2013improved}, there also exist randomized algorithms, some of which have an even smaller runtime, but whose outputs are only correct with a certain probability, usually below $90 \%$, see, e.g., \cite{hassanieh2012SODA,pawar}. For further methods we refer to the survey \cite{gilbert2014}.

To the best of our knowledge there exist no sparse fast algorithms for trigonometric transforms.
However, besides the DFT, the discrete cosine transform (DCT) is one of the most widely used algorithms in mathematics, engineering and data processing.  One reason for the interest in DCT algorithms is their capability  to decorrelate signals governed by Markov processes approaching  the statistically optimal Karhunen-Loeve transform (KLT), see \cite{RY90}.
The DCT is one of the major components  in data compression. As for the DFT there exist many different fast algorithms for the DCT, based either on FFT in complex arithmetic or directly on real arithmetic, see \cite{plonka_dct}.
One possible application of the sparse DCT  is the fast evaluation of polynomials in monomial form from sparse  expansions of Chebyshev polynomials, see, e.g., \cite{PPST}, Chapter 6. 
\medskip

In this paper we present a first sparse fast algorithm for the inverse DCT-II (or equivalently for the DCT-III), assuming that the resulting vector ${\x} \in {\mathbb R}^{N}$ possesses a one-block (or short) support of length $m < N$, where $m$ does not have to be known a priori. This algorithm has a runtime of ${\mathcal O}\left(m \log m  \log \frac{2N}{m}\right)$ and requires only $\mathcal{O}\left(m\log \frac{2N}{m}\right)$ samples.

\subsection{Notation and Problem Statement}\label{sec:notation}

Let $N=2^{J-1}$ with $J \geq2$. We say that a vector $\x= (x_{j})_{j=0}^{N-1}\in\RR^N$ possesses a \emph{one-block support} $S^{\x}$ \emph{of length} $m$ if 
\[
 x_j=0 \qquad \forall\, j\notin S^{\x}\coloneqq I_{\mu^{\x},\nu^\x} \coloneqq\left\{\mu^\x,\left(\mu^\x+1\right)\mods N,\dotsc,\nu^\x\right\},
\]
for some $\mu^\x\in\{0,\dotsc,N-1\}$ and $\nu^\x\coloneqq\left(\mu^\x+m-1\right)\mods N$, where we allow the support to be periodically wrapped  around $0$. The interval $S^{\x}\coloneqq I_{\mu^\x,\nu^\x}$ is called the \emph{support interval}, $\mu^\x$ the \emph{first support index} and $\nu^\x$ the \emph{last support index}.  
The support length $m$ is uniquely determined. If $m\leq N/2$, the first support index $\mu^\x$ is unique as well, which does not have to be the case for $m>N/2$. For example,
$\x = (1,0,0,0,1,0,0,0)^T$ has support length $m=5$, but the first support index $\mu^\x$ can be chosen to be $0$ or $4$.

Further, we consider  $\y \coloneqq(\x^T,(\J_N\x)^T)^T \in\RR^{2N}$ and say that $\y$ possesses a \emph{reflected block support of block length} $m$ if $\x\in\RR^N$ has a one-block support of length $m$. Here, $\J_{N} \in {\mathbb R}^{N \times N}$ denotes the counter identity, 
 \[
  \J_N\coloneqq\left(\delta_{k,\,N-1-l}\right)_{k,\,l=0}^{N-1}=
               \begin{pmatrix}
                0 & \dots & 0 & 1 \\                                                         
                0 & & 1 & 0 \\
                0 & \iddots & & \vdots \\
                1 & & 0 & 0 
               \end{pmatrix},
 \]
so the second half of $\y$ is the reflection of $\x$. 

For $N\in\NN$ the \emph{cosine matrix of type II} is defined as
\begin{equation*}
  \C\coloneqq\sqrt{\frac{2}{N}}\left(\eps_N(k)\cos\left(\frac{k(2l+1)\pi}{2N}\right)\right)_{k,\,l=0}^{N-1},
\end{equation*}
 where $\eps_N(k)\coloneqq \frac{1}{\sqrt{2}}$ for $k\equiv 0\mods N $ and $\eps_N(k)\coloneqq 1$ for $k\not\equiv 0\mods N $.
This matrix is orthogonal, i.e., $ \C \left(\C\right)^{T} = {\mathbf I}_{N}$, where ${\mathbf I}_{N}$ denotes  the identity matrix of size $N \times N$. 
The \emph{discrete cosine transform of type II (DCT-II)} of $\x\in\RR^N$ is given by
 \[
  \cxx\coloneqq\C\x.
 \]
The inverse DCT-II is equal to the \emph{discrete cosine transform of type III (DCT-III)} with transformation matrix ${\mathbf C}^{\mathrm{III}} \coloneqq \left(\C\right)^{T}$.
The \emph{Fourier matrix} of size $N \times N$ is defined as $\F_N\coloneqq\left(\omega_N^{kl}\right)_{k,\,l=0}^{N-1}$, where $\omega_N\coloneqq e^{\frac{-2\pi i}{N}}$ is an $N$th primitive root of unity. Then the \emph{discrete Fourier transform (DFT)} of $\x\in\CC^N$ is given by 
\[
 \widehat{\x}\coloneqq \F_N\x.
\]
The purpose of this paper is to derive a sparse fast DCT algorithm for computing $\x$ with (unknown) one-block support of length $m <N$ from its DCT-transformed vector $\cxx$ in sublinear time ${\mathcal O}\left(m \log m  \log \frac{2N}{m}\right)$. In particular, if $m$ approaches the vector length $N$, the proposed algorithm still has a runtime complexity of ${\mathcal O}(N \log N)$ that is also achieved by fast DCT algorithms for vectors with full support.
Our algorithm is truly deterministic and returns the correct vector $\x$ if we assume that the vector $\y=(\x^T,(\J_N\x)^T)^T \in {\mathbb R}^{2N}$ satisfies the condition that for each nonzero entry $y_{k} \neq 0$ of $\y$ we also have that 
\begin{equation}\label{eq:suppose}
 y_{k \mods 2^{j}}^{(j)} \coloneqq \sum_{l=0}^{2^{J-j}-1} y_{k+2^j l}  \neq 0 \qquad \forall\, j \in\{0, \ldots , J\}. 
\end{equation}
This assumption is for example satisfied if all nonzero entries of $\x$ are positive or all nonzero entries are negative, i.e., $\x\in\RR^N_{\geq0}$ or $\x\in\RR_{\leq0}^N$.
In practice, for noisy data, one has to ensure that, for a threshold $\eps>0$ depending on the noise level,
 \[
  \left|y\jj_{k\mods 2^j}\right|>\eps \qquad\forall\,j\in\{0,\dotsc,J\}.
 \]

\subsection {Outline of the Paper}
The algorithm presented in this paper is based on the fact that the DCT-II and the DFT are closely related. In Section \ref{sec:dct_dft} we recall that the DCT-II can be computed by applying the DFT to the vector $\y \coloneqq(\x^T,(\J_N\x)^T)^T \in\RR^{2N}$ of double length. We then derive a sparse fast DFT algorithm for vectors $\y$ with reflected block support, generalizing ideas in \cite{plonka_smallsupp} and \cite{plonka_nonneg}. This sparse FFT does not require a priori knowledge of the length $m$ of the two support blocks of $\y$. 

The vector $\y$ is computed using an iterative procedure. As in \cite{plonka_smallsupp,plonka_nonneg} we use the $2^{j}$-length periodizations $\y^{(j)}$ of $\y \in {\mathbb R}^{2^{J}}$ for $j\in\{0, \ldots , J-1\}$ and efficiently compute $\y^{(j+1)}$ exploiting that $\y^{(j)}$ is known. 
This approach requires a detailed investigation of the support properties of $\y^{(j+1)}$, depending on the support of $\y^{(j)}$. The corresponding observations are summarized in Section \ref{sec:support}.
We have to distinguish two different cases for the reconstruction of $\y^{(j+1)}$ from $\y^{(j)}$, namely that $\y^{(j)}$ possesses either a one-block support of length $m^{(j)} \leq 2m$ or a two-block support, where each block already has length $m$ and the second block is the reflection of the first.
We present two numerical procedures to treat these cases in Section \ref{sec:dft_procedures}.
Using these methods we develop the complete sparse FFT algorithm for computing $\y$ and prove its arithmetical complexity in Section \ref{sec:dft_alg}. Section \ref{sec:detection} is devoted to the stable, efficient detection of the support of the periodized vectors $\y^{(j+1)}$. The sparse DCT algorithm, presented in Section \ref{sec:dct_alg}, can now simply be derived from this special sparse FFT. 
Finally, we extensively test the new sparse FFT and  the sparse DCT algorithm with respect to runtime and stability for noisy input data in Section \ref{sec:numerics}. 

\section{DCT-II via DFT}\label{sec:dct_dft}
A fast DCT algorithm can be obtained by employing fast DFT algorithms. Recall that for $\x = \left(x_{k}\right)_{k=0}^{N-1} \in\RR^N$ we defined $\y = \left(y_{k}\right)_{k=0}^{2N-1}\in\RR^{2N}$ as $\y\coloneqq(\x^{T}, (\J_{N}\x)^{T})^{T}$ in Section \ref{sec:notation}, so $\y$ can be written as
\begin{equation}\label{eq:y}
 y_k\coloneqq\begin{cases}
              x_k, & k\in\{0,\dotsc,N-1\}, \\
              x_{N-1-k}, & k\in\{N,\dotsc,2N-1\}.
             \end{cases}
\end{equation}
The following lemma shows the close relation between $\cxx$ and $\widehat{\y}$, namely that $\cxx$ can be computed from $\widehat{\y}$ and vice versa, see \cite{RY90}.
\begin{lem}\label{lem:dctx_from_yhat}
 Let $\x\in\RR^N$ and $\y=(\x^T,(\J_N\x)^T)^T$. Then $\cxx = \left(\cx_{k}\right)_{k=0}^{N-1} = \C \x$ is given by
\begin{equation}\label{eq:xy}
  \cx_k=\frac{\eps_N(k)}{\sqrt{2N}}\omega^k_{4N}\cdot\widehat{y}_k, \qquad k\in\{0, \ldots , N-1\},
\end{equation}
 where $\widehat{\y} = \left(\widehat{y}_{k}\right)_{k=0}^{2N-1} = \F_{2N} \y$.
Conversely, $\widehat{\y}$  is completely determined by $\cxx$ via 
\begin{equation}\label{converse}
  \widehat{y}_k=\begin{cases}
                 \frac{\sqrt{2N}}{\eps_N(k)}\omega_{4N}^{-k}\cdot\cx_k, & k\in\{0,\dotsc,N-1\}, \\
                 0, & k=N, \\
                 -\frac{\sqrt{2N}}{\eps_N(2N-k)}\omega_{4N}^{-k}\cdot\cx_{2N-k}, & k\in\{N+1,\dotsc,2N-1\}.
                \end{cases}
\end{equation}
\end{lem}
\begin{proof}
1.  Let $k\in\{0,\dotsc,N-1\}$. We find that
 \begin{align*}
  \cx_k&=\sqrt{\frac{2}{N}}\eps_N(k)\sum_{l=0}^{N-1}\cos\left(\frac{2\cdot k(2l+1)\pi}{2\cdot2N}\right)x_l 
  =\frac{\eps_N(k)}{\sqrt{2N}}\sum_{l=0}^{N-1}\left(\omega_{4N}^{k(2l+1)}+\omega_{4N}^{-k(2l+2-1)}\right)x_l \\
  &=\frac{\eps_N(k)}{\sqrt{2N}}\omega_{4N}^k\sum_{l=0}^{N-1}\left(\omega_{2N}^{kl}+\omega_{2N}^{k(2N-1-l)}\right)x_l \\
  &=\frac{\eps_N(k)}{\sqrt{2N}}\omega_{4N}^k\left(\sum_{l=0}^{N-1}\omega_{2N}^{kl}x_l+\sum_{l'=N}^{2N-1}\omega_{2N}^{kl'}x_{2N-1-l'}\right) 
  =\frac{\eps_N(k)}{\sqrt{2N}}\omega_{4N}^k\cdot\widehat{y}_k. \notag
 \end{align*} 
2. For $k\in\{0,\dotsc,N-1\}$  the claim in (\ref{converse}) follows directly from (\ref{eq:xy}). For $k\in\{N,\dotsc,2N-1\}$ note that by construction $\y$ is symmetric, 
\begin{equation}\label{eq:sym}
  \J_{2N}\y=\J_{2N}\begin{pmatrix}
                    \x \\
                    \J_N\x
                   \end{pmatrix}
                   =\begin{pmatrix}
                     \x \\
                     \J_N\x
                    \end{pmatrix}=\y,
\end{equation}
 which implies that
 \begin{align}\label{eq:useful}
  \widehat{y}_k=\left(\widehat{\J_{2N}\y}\right)_k&=\sum_{l'=0}^{2N-1}\omega_{2N}^{kl'}y_{2N-1-l'} 
  =\sum_{l=0}^{2N-1}\omega_{2N}^{k(2N-1-l)}y_l 
  =\omega_{2N}^{-k}\sum_{l=0}^{2N-1}\omega_{2N}^{-kl}y_l.
 \end{align}
If $k=N$, we obtain from (\ref{eq:useful}) that
 \[
  \widehat{y}_N=\omega_{2N}^{-N}\sum_{l=0}^{2N-1}\omega_{2N}^{-Nl}y_l
  =-\sum_{l=0}^{2N-1}\omega_{2N}^{Nl}y_l=-\widehat{y}_N,
 \]
 so $\widehat{y}_N=0$. If $k\in\{N+1,\dotsc,2N-1\}$, then $2N-k\in\{1,\dotsc,N-1\}$, and (\ref{eq:useful}) yields
 \begin{align*}
  \widehat{y}_k&=\omega_{2N}^{-k}\sum_{l=0}^{2N-1}\omega_{2N}^{-kl} y_l
  =\omega_{2N}^{-k}\sum_{l=0}^{2N-1}\omega_{2N}^{(2N-k)l}y_l \\
  &=\omega_{2N}^{-k}\cdot\widehat{y}_{2N-k} =-\frac{\sqrt{2N}}{\eps_N(2N-k)}\omega_{4N}^{-k}\cdot\cx_{2N-k}.
 \end{align*}
\end{proof}
 
\section{Support Properties of the Periodized Vectors}\label{sec:support}
We want to develop a deterministic sublinear-time DCT algorithm for reconstructing $\x \in {\mathbb R}^{N}$  with one-block support from $\cxx$. Using the close relation between the DFT and the DCT-II shown in Lemma \ref{lem:dctx_from_yhat}, this problem can be transformed into deriving a sparse FFT algorithm for reconstructing $\y = ( \x^{T}, (\J_{N} \x)^{T})^{T}$, which has a reflected block support, from $\widehat{\y}$. For this purpose we extend recent approaches in \cite{plonka_smallsupp,plonka_nonneg} for FFT reconstruction of vectors with one-block support to our setting. 

We assume that $\y$ satisfies (\ref{eq:suppose}) in order to avoid cancellation of nonzero entries in the iterative algorithm.
Similarly as in \cite{plonka_smallsupp,plonka_nonneg} we employ a divide-and-conquer technique to efficiently compute $\y$ from $\widehat{\y}$. 

Let $N\coloneqq2^{J-1}$ for some $J\geq2$. For $j\in\{0,\dotsc,J\}$ let the \emph{periodization} $\y^{(j)} \in\RR^{2^j}$ of $\y =(y_{k})_{k=0}^{2^{J}-1}$ be given as in (\ref{eq:suppose}), i.e.,
\begin{equation}\label{eq:per_sum}
  \y^{(j)}=\left(y_k^{(j)}\right)_{k=0}^{2^j-1} \coloneqq\left(\sum_{l=0}^{2^{J-j}-1}y_{k+2^jl}\right)_{k=0}^{2^j-1}.
\end{equation}
In particular, we observe that $\y^{(J)}=\y$,  $\y^{(0)}=\sum\limits_{k=0}^{2N-1}y_k$ and  
\begin{equation} \label{eq:per_el}  
 y^{(j)}_{k} = y_{k}^{(j+1)}+ y_{k+2^{j}}^{(j+1)}, \qquad k\in\left\{0, \ldots , 2^{j}-1\right\},
\end{equation}
for $j \in \{0, \ldots , J-1 \}$. 
The following lemma shows that the Fourier transform $\widehat{\y\jj}$ of $\y^{(j)}$ is obtained by equidistantly sampling $2^{j}$ entries of $\widehat{\y}$, and that the periodization $\y^{(j)}$ possesses a symmetry property analogous to the one of $\y$ in (\ref{eq:sym}).
\begin{lem}\label{lem:periodization}
Let $N=2^{J-1}$, $\y\in\RR^{2N}$ and $j\in\{0,\dotsc,J\}$. Then $\widehat{\y\jj}$ satisfies
\begin{equation}\label{eq:fourierj}
  \widehat{\y^{(j)}} \coloneqq \F_{2^{j}} \y^{(j)}=\left(\widehat{y}_{2^{J-j}k}\right)_{k=0}^{2^j-1}.
\end{equation}
 Further, $\y\jj$ is symmetric, i.e., 
\begin{equation}\label{eq:symmetry}  
 \y^{(j)} = \J_{2^{j}} \y^{(j)}. 
\end{equation}
\end{lem}
\begin{proof} 
In \cite{plonka_smallsupp}, Lemma 2.1, (\ref{eq:fourierj}) was already shown.
By (\ref{eq:sym}) we have that $y_{k} = y_{2^{J}-1-k}$ for $k \in \left\{0, \ldots ,2^{J}-1\right \}$. It follows for the entries of $\y^{(J-1)}$ that 
\[
 y_{k}^{(J-1)} = y_{k} + y_{k+2^{J-1}} = y_{2^{J}-1-k} + y_{2^{J}-1-(k+2^{J-1}) } 
=  y_{2^{J-1}-1-k}^{(J-1)} 
\]
for all $k\in\left\{0,\dotsc,2^{J-1}-1\right\}$; thus $\y^{(J-1)} = \J_{2^{J-1}} \y^{(J-1)}$. For $j< J-1$ (\ref{eq:symmetry}) can be proven inductively with the same argument.
\end{proof}

Lemma 2.2 in \cite{plonka_smallsupp} shows how the Fourier transform of a vector $\uu \in {\mathbb R}^{2^{j+1}}$, $j\ge 0$, changes if its entries are cyclically shifted by $2^j$.
\begin{lem}\label{lem:shift}
Let $\uu\in\RR^{2^{j+1}}$, $j\geq0$, and let the shifted vector $\uu^1 \coloneqq \left(u_{k}^{1}\right)_{k=0}^{2^{j+1}-1}\in\RR^{2^{j+1}}$ be given by 
 \[
  u^1_{k} \coloneqq u_{\left(k+2^j\right)\mods 2^{j+1}}, \qquad k\in\left\{0,\dotsc,2^{j+1}-1\right\}.
 \]
Then $\widehat{\uu^1}$ satisfies
 \[
  \widehat{u^1}_k=(-1)^{k}\widehat{u}_k \qquad\forall\, k\in\left\{0,\dotsc,2^{j+1}-1\right\}.
 \]
\end{lem}
Our goal is to reconstruct $\y$ from $\widehat{\y}$ by successively computing its periodizations $\y^{(0)}$, $\y^{(1)},\dotsc,\y^{(J)}=\y$.  In the $j$th iteration step of the procedure we have to determine $\y^{(j+1)}$ from $\widehat{\y\jjj}$ efficiently, which can be done by employing the vector $\y\jj$ known from the previous step. In order for this approach to work we need to investigate how the support blocks of $\y\jjj$ can look like if the support of $\y\jj$ is given.

We start by inspecting the support properties of $\y^{(J)} = \left(\x^{T}, (\J_N \x)^{T}\right)^{T} \in {\mathbb R}^{2N}$.
\begin{lem}\label{lem:supporty}
Let $\x \in {\mathbb R}^{N}$ with $N=2^{J-1}$ have the one-block support $S^{\x}=I_{\mu^\x,\nu^\x}$ of length $m<N$. Set $\y= \y^{(J)} = (\x^{T}, (\J_N \x)^{T})^{T} \in {\mathbb R}^{2N}$ and assume that $\y$ satisfies $(\ref{eq:suppose})$.
\begin{enumerate}[label=\upshape\roman*)]
\item\label{item:supporty_i} If $\mu^\x \le \nu^\x$, $\y$ has the (reflected) two-block support $S^{(J)}=I_{\mu^\x,\nu^\x}\cup I_{2N-1-\nu^\x, 2N-1-\mu^\x}$ with two blocks of the same length $m=\nu^\x-\mu^\x+1$.
In the special cases $\mu^\x=0$ or $\nu^\x=N-1$, these two support blocks are adjacent and form a (reflected) one-block support, namely $S^{(J)}=I_{2N-1-\nu^\x, \nu^\x}$ for $\mu^\x=0$ and $S^{(J)}=I_{\mu^\x,2N-1-\mu^\x}$ for $\nu^\x=N-1$.
\item\label{item:supporty_ii} If $\mu^\x > \nu^\x$, then $\y$ has the (reflected) two-block support $S^{(J)}=I_{\mu^\x, 2N-1-\mu^\x} \cup I_{2N-1-\nu^\x, \nu^\x}$, where the two separated support blocks have the possibly different lengths $2(N-\mu^\x)$ and $2(\nu^\x+1)$, which sum up to $2m=2\left(N-\mu^\x+\nu^\x+1\right)$.
\end{enumerate}
\end{lem}

The proof of Lemma \ref{lem:supporty} follows directly from the definition of $\y$. Note that since $m<N$, $\mu^\x = \nu^\x+1$ is not possible in \ref{item:supporty_ii}. Therefore, the two support blocks in \ref{item:supporty_ii} are always separated. For algorithmic purposes we then denote the \emph{first index of the block centered around the middle of the vector} by $\mu^{(J)}$ and the \emph{first index of the block centered around the boundary of the vector} by $\eta^{(J)}$. See Figure \ref{fig2} (left) for a visualization. 

\def\a{2.4}
\def\b{\a/4}
\def\d{0.1}
\def\dd{3/4*\d}
\begin{figure}[!ht]
\tiny
\begin{alignat*}{4}
&\y^{(j)} \quad 
&& \begin{tikzpicture}[baseline=-0.75ex]
       \tikzset{
    brace/.style={decoration={brace},decorate},
    every pin edge/.style={thin}
 }
 \draw (0,0) -- (\a,0);
 \draw (0,\d) -- (0,-\d) node[below] {0};
 \draw (\a,\d) -- (\a,-\d) node[below] {$2^j-1$} ;
 \draw (\a/2,\d) -- (\a/2,-\d) node[below] {$2^{j-1}-1$};
 \node (1start) at (\a/2-3/4*\b,3/4*\d) {};
 \node (1end) at (\a/2+3/4*\b,3/4*\d) {};
 \draw [brace] (1start.north) -- node [pos=0.5, above] {$m^{(j)}$} (1end.north);  
 \draw[blue, thick, pattern=north east lines, pattern color=blue] (\a/2-3/4*\b,3/4*\d) -- (\a/2-3/4*\b,-3/4*\d) -- (\a/2+1/4*\b,-3/4*\d) -- (\a/2+1/4*\b,3/4*\d) -- (\a/2-3/4*\b,3/4*\d);
 \draw[orange, thick, pattern=north west lines, pattern color=orange] (\a/2-1/4*\b,-3/4*\d) -- (\a/2-1/4*\b,3/4*\d) -- (\a/2+3/4*\b,3/4*\d) -- (\a/2+3/4*\b,-3/4*\d) -- (\a/2-1/4*\b,-3/4*\d);
 \node [inner sep=3pt,pin={[inner sep=2pt, pin distance=0.10cm]245:$\mu\jj$}] at (\a/2-3/4*\b,-3/4*\d) {};
\end{tikzpicture} 
\qquad
&&\y^{(j)} \quad 
&& \begin{tikzpicture}[baseline=-0.75ex]
\tikzset{every pin edge/.style={thin}}
 \draw (0,0) -- (\a,0);
 \draw (0,\d) -- (0,-\d) node[below] {0};
 \draw (\a,\d) -- (\a,-\d) node[below] {$2^j-1$} ;
 \draw (\a/2,\d) -- (\a/2,-\d) node[below] {$2^{j-1}-1$};
 \draw[blue, thick, pattern=north east lines, pattern color=blue] (\a-3/4*\b,3/4*\d) -- (\a-3/4*\b,-3/4*\d) -- (\a,-3/4*\d) -- (\a,3/4*\d) -- (\a-3/4*\b,3/4*\d);
 \draw[blue, thick, pattern=north east lines, pattern color=blue] (0,3/4*\d) -- (0,-3/4*\d) -- (1/4*\b,-3/4*\d) -- (1/4*\b,3/4*\d) -- (0,3/4*\d);
 \draw[orange, thick, pattern=north west lines, pattern color=orange] (\a-1/4*\b,-3/4*\d) -- (\a-1/4*\b,3/4*\d) -- (\a,3/4*\d) -- (\a,-3/4*\d) -- (\a-1/4*\b,-3/4*\d);
 \draw[orange, thick, pattern=north west lines, pattern color=orange] (0,-3/4*\d) -- (0,3/4*\d) -- (3/4*\b,3/4*\d) -- (3/4*\b,-3/4*\d) -- (0,-3/4*\d);
 \node [inner sep=3pt,pin={[inner sep=2pt, pin distance=0.15cm]270:$\mu^{(j)}$}] at (\a-3/4*\b,-3/4*\d) {};
\end{tikzpicture} \\
& \y^{(j+1)} \quad 
&& \begin{tikzpicture}[baseline=-0.75ex]
        \tikzset{
    brace/.style={decoration={brace},decorate},
    every pin edge/.style={thin}
 }
 \draw (0,0) -- (2*\a,0);
 \draw (0,\d) -- (0,-\d) node[below] {0};
 \draw (2*\a,\d) -- (2*\a,-\d) node[below] {$2^{j+1}-1$};
 \draw (\a/2,\d) -- (\a/2,-\d);
 \draw (\a,\d) -- (\a,-\d) node[below] {$2^j-1$};
 \draw (3*\a/2,\d) -- (3*\a/2,-\d);
 \node (1start) at (\a/2-3/4*\b,3/4*\d) {};
 \node (1end) at (\a/2+1/4*\b,3/4*\d) {};  
 \node (2start) at (\a+\a/2-1/4*\b,3/4*\d) {};
 \node (2end) at (\a+\a/2+3/4*\b,3/4*\d) {};
 \draw [brace] (1start.north) -- node [pos=0.5, above] { $m$} (1end.north);
 \draw [brace] (2start.north) -- (2end.north) node [pos=0.5, above] {$m$};
 \draw[blue, thick, pattern=north east lines, pattern color=blue] (\a/2-3/4*\b,3/4*\d) -- (\a/2-3/4*\b,-3/4*\d) -- (\a/2+1/4*\b,-3/4*\d) -- (\a/2+1/4*\b,\dd) -- (\a/2-3/4*\b,\dd); 
 \draw[orange, thick, pattern=north west lines, pattern color=orange] (\a+\a/2-1/4*\b,-\dd)  -- (\a+\a/2-1/4*\b,\dd) -- (\a+\a/2+3/4*\b,\dd) -- (\a+\a/2+3/4*\b,-\dd) -- (\a+\a/2-1/4*\b,-\dd);
 \node [inner sep=3pt,pin={[inner sep=2pt, pin distance=0.15cm]270:$\mu\jjj=\mu\jj$}] at (\a/2-3/4*\b,-\dd) {};
\end{tikzpicture}
\qquad
&& \y^{(j+1)} \quad
&& \begin{tikzpicture}[baseline=-0.75ex]
        \tikzset{
    brace/.style={decoration={brace},decorate},
    every pin edge/.style={thin}
 }
 \draw (0,0) -- (2*\a,0);
 \draw (0,\d) -- (0,-\d) node[below] {0};
 \draw (2*\a,\d) -- (2*\a,-\d) node[below] {$2^{j+1}-1$};
 \draw (\a/2,\d) -- (\a/2,-\d);
 \draw (\a,\d) -- (\a,-\d) node[below] {$2^j-1$};
 \draw (3*\a/2,\d) -- (3*\a/2,-\d);
 \draw[blue, thick, pattern=north east lines, pattern color=blue] (\a-3/4*\b,\dd) -- (\a-3/4*\b,-\dd) -- (\a+1/4*\b,-\dd) -- (\a+1/4*\b,\dd) -- (\a-3/4*\b,\dd); 
 \draw[orange, thick, pattern=north west lines, pattern color=orange] (\a-1/4*\b,-\dd) -- (\a-1/4*\b,\dd) -- (\a+3/4*\b,\dd) -- (\a+3/4*\b,-\dd) -- (\a-1/4*\b,-\dd);
 \node [inner sep=3pt,pin={[pin distance=0.1cm]210:$\mu\jjj=\mu^{(j)}$}] at (\a-3/4*\b,-\dd) {};
\end{tikzpicture} \\
 \text{or} \quad &\y^{(j+1)} \quad
&& \begin{tikzpicture}[baseline=-0.75ex]
        \tikzset{
    brace/.style={decoration={brace},decorate},
    every pin edge/.style={thin}
 }
 \draw (0,0) -- (2*\a,0);
 \draw (0,\d) -- (0,-\d) node[below] {0};
 \draw (2*\a,\d) -- (2*\a,-\d) node[below] {$2^{j+1}-1$};
 \draw (\a/2,\d) -- (\a/2,-\d);
 \draw (\a,\d) -- (\a,-\d) node[below] {$2^j-1$};
 \draw (3*\a/2,\d) -- (3*\a/2,-\d);
 \draw[blue, thick, pattern=north east lines, pattern color=blue] (\a+\a/2-3/4*\b,\dd) -- (\a+\a/2-3/4*\b,-\dd) -- (\a+\a/2+1/4*\b,-\dd) -- (\a+\a/2+1/4*\b,\dd) -- (\a+\a/2-3/4*\b,\dd); 
 \draw[orange, thick, pattern=north west lines, pattern color=orange] (\a/2-1/4*\b,-\dd) -- (\a/2-1/4*\b,\dd) -- (\a/2+3/4*\b,\dd) -- (\a/2+3/4*\b,-\dd) -- (\a/2-1/4*\b,-\dd);
 \node [inner sep=3pt,pin={[inner sep=2pt, pin distance=0.15cm]270:$2^j+\mu^{(j)}$}] at (\a+\a/2-3/4*\b,-\dd) {};
\end{tikzpicture} 
\qquad
 \text{or} \quad &&\y^{(j+1)} \quad
&& \begin{tikzpicture}[baseline=-0.75ex]
        \tikzset{
    brace/.style={decoration={brace},decorate},
    every pin edge/.style={thin}
 }
 \draw (0,0) -- (2*\a,0);
 \draw (0,\d) -- (0,-\d) node[below] {0};
 \draw (2*\a,\d) -- (2*\a,-\d) node[below] {$2^{j+1}-1$};
 \draw (\a/2,\d) -- (\a/2,-\d);
 \draw (\a,\d) -- (\a,-\d) node[below] {$2^j-1$};
 \draw (3*\a/2,\d) -- (3*\a/2,-\d);
 \draw[blue, thick, pattern=north east lines, pattern color=blue] (0,\dd) -- (0,-\dd) -- (1/4*\b,-\dd) -- (1/4*\b,\dd) -- (0,\dd); 
 \draw[orange, thick, pattern=north west lines, pattern color=orange] (0,-\dd)  -- (0,\dd) -- (3/4*\b,\dd) -- (3/4*\b,-\dd) -- (0,-\dd);
 \draw[blue, thick, pattern=north east lines, pattern color=blue] (2*\a-3/4*\b,\dd) -- (2*\a-3/4*\b,-\dd) -- (2*\a,-\dd) -- (2*\a,\dd) -- (2*\a-3/4*\b,\dd); 
 \draw[orange, thick, pattern=north west lines, pattern color=orange] (2*\a-1/4*\b,-\dd)  -- (2*\a-1/4*\b,\dd) -- (2*\a,\dd) -- (2*\a,-\dd) -- (2*\a-1/4*\b,-\dd);
 \node [inner sep=3pt,pin={[inner sep=2pt, pin distance=0.15cm]265:$\mu^{(j+1)}=2^j+\mu^{(j)}$}] at (2*\a-3/4*\b,-\dd) {};
\end{tikzpicture}
\end{alignat*}
\caption{Illustration of the two possibilities for the support of $\y^{(j+1)}$ for given $\y^{(j)}$ according to Theorem \ref{thm:supp}  in case \ref{item:thm1_middle} (left) and case \ref{item:thm1_bound} (right).}
\label{fig1}
\end{figure}
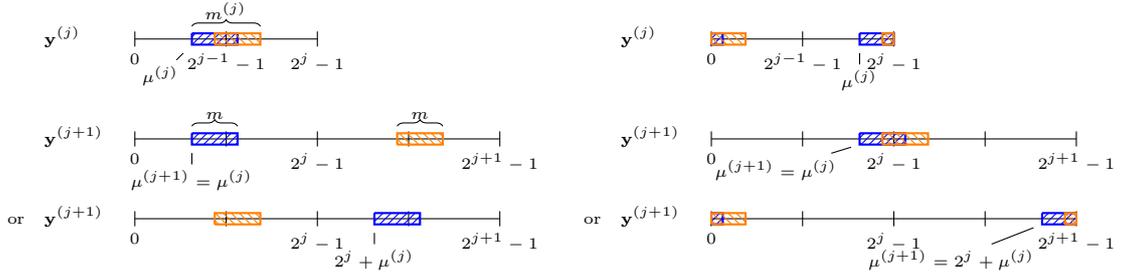
\begin{figure}[!ht]
\tiny
\begin{alignat*}{4}
&\y^{(J-1)} \quad 
&& \begin{tikzpicture}[baseline=-0.75ex]
      \tikzset{
    brace/.style={decoration={brace},decorate},
    every pin edge/.style={thin}
 }
 \draw (0,0) -- (\a,0);
 \draw (0,\d) -- (0,-\d) node[below] {0};
 \draw (\a,\d) -- (\a,-\d) node[below] {$2^{J-1}-1$} ;
 \draw (\a/2,\d) -- (\a/2,-\d) node[below] {$2^{J-2}-1$};
 \node (1start) at (0,\dd) {};
 \node (1end) at (3/4*\b,\dd) {};
 \node (2start) at (\a-3/4*\b,\dd) {};
 \node (2end) at (\a,\dd) {};  
 \draw[blue, thick, pattern=north east lines, pattern color=blue] (\a-3/4*\b,\dd) -- (\a-3/4*\b,-\dd) -- (\a,-\dd) -- (\a,\dd) -- (\a-3/4*\b,\dd);
 \draw[blue, thick, pattern=north east lines, pattern color=blue] (0,\dd) -- (0,-\dd) -- (1/4*\b,-\dd) -- (1/4*\b,\dd) -- (0,\dd);
 \draw[orange, thick, pattern=north west lines, pattern color=orange] (\a-1/4*\b,-\dd) -- (\a-1/4*\b,\dd) -- (\a,\dd) -- (\a,-\dd) -- (\a-1/4*\b,-\dd);
 \draw[orange, thick, pattern=north west lines, pattern color=orange] (0,-\dd) -- (0,\dd) -- (3/4*\b,\dd) -- (3/4*\b,-\dd) -- (0,-\dd);
 \node [inner sep=3pt,pin={[inner sep=2pt, pin distance=0.2cm]267:$\mu^{(J-1)}$}] at (\a-3/4*\b,-\dd) {};
\end{tikzpicture} 
\qquad 
&&\y^{(j)} \quad 
&& \begin{tikzpicture}[baseline=-0.75ex]
\tikzset{every pin edge/.style={thin}}
 \draw (0,0) -- (\a,0);
 \draw (0,\d) -- (0,-\d) node[below] {0};
 \draw (\a,\d) -- (\a,-\d) node[below] {$2^j-1$} ;
 \draw (\a/2,\d) -- (\a/2,-\d) node[below] {$2^{j-1}-1$};
 \draw[blue, thick, pattern=north east lines, pattern color=blue] (\a/4-3/4*\b,\dd) -- (\a/4-3/4*\b,-\dd) -- (\a/4+1/4*\b,-\dd) -- (\a/4+1/4*\b,\dd) -- (\a/4-3/4*\b,\dd);
 \draw[orange, thick, pattern=north west lines, pattern color=orange] (3/4*\a-1/4*\b,-\dd) -- (3/4*\a-1/4*\b,\dd) -- (3/4*\a+3/4*\b,\dd) -- (3/4*\a+3/4*\b,-\dd) -- (3/4*\a-1/4*\b,-\dd);
 \node [inner sep=3pt,pin={[inner sep=2pt, pin distance=0.1cm]275:$\mu^{(j)}$}] at (\a/4-3/4*\b,-\dd) {};
\end{tikzpicture} \\
& \y^{(J)} \quad 
&& \begin{tikzpicture}[baseline=-0.75ex]
      \tikzset{
    brace/.style={decoration={brace},decorate},
    every pin edge/.style={thin}
    }
 \draw (0,0) -- (2*\a,0);
 \draw (0,\d) -- (0,-\d) node[below] {0};
 \draw (2*\a,\d) -- (2*\a,-\d) node[below] {$2^{J}-1$};
 \draw (\a/2,\d) -- (\a/2,-\d);
 \draw (\a,\d) -- (\a,-\d) node[below] {$2^{J-1}-1$};
 \draw (3*\a/2,\d) -- (3*\a/2,-\d);
 \node (1start) at (0,\dd) {};
 \node (1end) at (1/4*\b,\dd) {}; 
 \node (2start) at (\a-3/4*\b,\dd) {};
 \node (2end) at (\a+3/4*\b,\dd) {};
 \node (3start) at (2*\a-1/4*\b,\dd) {};
 \node (3end) at (2*\a,\dd) {};
 \draw[blue, thick, pattern=north east lines, pattern color=blue] (0,\dd) -- (0,-\dd) -- (1/4*\b,-\dd) -- (1/4*\b,\dd) -- (0,\dd); 
 \draw[orange, thick, pattern=north west lines, pattern color=orange] (\a,-\dd)  -- (\a,\dd) -- (\a+3/4*\b,\dd) -- (\a+3/4*\b,-\dd) -- (\a,-\dd);
 \draw[blue, thick, pattern=north east lines, pattern color=blue] (\a-3/4*\b,\dd) -- (\a-3/4*\b,-\dd) -- (\a,-\dd) -- (\a,\dd) -- (\a-3/4*\b,\dd); 
 \draw[orange, thick, pattern=north west lines, pattern color=orange] (2*\a-1/4*\b,-\dd)  -- (2*\a-1/4*\b,\dd) -- (2*\a,\dd) -- (2*\a,-\dd) -- (2*\a-1/4*\b,-\dd);
 \node [inner sep=3pt,pin={[inner sep=2pt, pin distance=0.1cm]215:$\eta^{(J)}$}] at (2*\a-1/4*\b,-\dd) {};
 \node [inner sep=3pt,pin={[inner sep=2pt, pin distance=0.05cm]265:$\mu^{(J)}=\mu^{(J-1)}$}] at (\a-3/4*\b,-\dd) {};
\end{tikzpicture} 
\qquad
&& \y^{(j+1)} \quad 
&& \begin{tikzpicture}[baseline=-0.75ex]
\tikzset{every pin edge/.style={thin}}
 \draw (0,0) -- (2*\a,0);
 \draw (0,\d) -- (0,-\d) node[below] {0};
 \draw (2*\a,\d) -- (2*\a,-\d) node[below] {$2^{j+1}-1$};
 \draw (\a/2,\d) -- (\a/2,-\d);
 \draw (\a,\d) -- (\a,-\d) node[below] {$2^j-1$};
 \draw (3*\a/2,\d) -- (3*\a/2,-\d);
 \draw[blue, thick, pattern=north east lines, pattern color=blue] (\a/4-3/4*\b,\dd) -- (\a/4-3/4*\b,-\dd) -- (\a/4+1/4*\b,-\dd) -- (\a/4+1/4*\b,\dd) -- (\a/4-3/4*\b,\dd); 
 \draw[orange, thick, pattern=north west lines, pattern color=orange] (\a+3/4*\a-1/4*\b,-\dd) -- (\a+3/4*\a-1/4*\b,\dd) -- (\a+3/4*\a+3/4*\b,\dd) -- (\a+3/4*\a+3/4*\b,-\dd) -- (\a+3/4*\a-1/4*\b,-\dd);
 \node [inner sep=3pt,pin={[inner sep=2pt, pin distance=0.1cm]275:$\mu^{(j)}$}] at (\a/4-3/4*\b,-\dd) {};
\end{tikzpicture} \\ 
\text{or} \quad & \y^{(J)} \quad 
&& \begin{tikzpicture}[baseline=-0.75ex]
      \tikzset{
    brace/.style={decoration={brace},decorate},
    every pin edge/.style={thin}
 }
 \draw (0,0) -- (2*\a,0);
 \draw (0,\d) -- (0,-\d) node[below] {0};
 \draw (2*\a,\d) -- (2*\a,-\d) node[below] {$2^{J}-1$};
 \draw (\a/2,\d) -- (\a/2,-\d);
 \draw (\a,\d) -- (\a,-\d) node[below] {$2^{J-1}-1$};
 \draw (3*\a/2,\d) -- (3*\a/2,-\d);
 \node (1start) at (0,\dd) {};
 \node (1end) at (3/4*\b,\dd) {}; 
 \node (2start) at (\a-1/4*\b,\dd) {};
 \node (2end) at (\a+1/4*\b,\dd) {};
 \node (3start) at (2*\a-3/4*\b,\dd) {};
 \node (3end) at (2*\a,\dd) {};
 \draw[orange, thick, pattern=north east lines, pattern color=orange] (0,\dd) -- (0,-\dd) -- (3/4*\b,-\dd) -- (3/4*\b,\dd) -- (0,\dd); 
 \draw[blue, thick, pattern=north west lines, pattern color=blue] (\a,-\dd)  -- (\a,\dd) -- (\a+1/4*\b,\dd) -- (\a+1/4*\b,-\dd) -- (\a,-\dd);
 \draw[orange, thick, pattern=north east lines, pattern color=orange] (\a-1/4*\b,\dd) -- (\a-1/4*\b,-\dd) -- (\a,-\dd) -- (\a,\dd) -- (\a-1/4*\b,\dd); 
 \draw[blue, thick, pattern=north west lines, pattern color=blue] (2*\a-3/4*\b,-\dd)  -- (2*\a-3/4*\b,\dd) -- (2*\a,\dd) -- (2*\a,-\dd) -- (2*\a-3/4*\b,-\dd);
 \node [inner sep=3pt,pin={[inner sep=2pt, pin distance=0.15cm]255:$\eta^{(J)}=2^{J-1}+\mu^{(J-1)}$}] at (2*\a-3/4*\b,-0.15) {};
\end{tikzpicture} 
\qquad
\text{or} \quad &&\y^{(j+1)} \quad
 && \begin{tikzpicture}[baseline=-0.75ex]
 \tikzset{every pin edge/.style={thin}}
 \draw (0,0) -- (2*\a,0);
 \draw (0,\d) -- (0,-\d) node[below] {0};
 \draw (2*\a,\d) -- (2*\a,-\d) node[below] {$2^{j+1}-1$};
 \draw (\a/2,\d) -- (\a/2,-\d);
 \draw (\a,\d) -- (\a,-\d) node[below] {$2^j-1$};
 \draw (3*\a/2,\d) -- (3*\a/2,-\d);
 \draw[blue, thick, pattern=north east lines, pattern color=blue] (\a+\a/4-3/4*\b,\dd) -- (\a+\a/4-3/4*\b,-\dd) -- (\a+\a/4+1/4*\b,-\dd) -- (\a+\a/4+1/4*\b,\dd) -- (\a+\a/4-3/4*\b,\dd); 
 \draw[orange, thick, pattern=north west lines, pattern color=orange] (3/4*\a-1/4*\b,-\dd) -- (3/4*\a-1/4*\b,\dd) -- (3/4*\a+3/4*\b,\dd) -- (3/4*\a+3/4*\b,-\dd) -- (3/4*\a-1/4*\b,-\dd);
 \node [inner sep=3pt,pin={[inner sep=2pt, pin distance=0.15cm]295:$2^j+\mu^{(j)}$}] at (\a+\a/4-3/4*\b,-\dd) {};
 \node [inner sep=3pt,pin={[inner sep=2pt, pin distance=0.15cm]255:$2^j-m\jj-\mu^{(j)}$}] at (\a-3/4*\b,-\dd) {};
\end{tikzpicture} 
 \end{alignat*}
\caption{Illustration of the two possibilities for the support of $\y^{(j+1)}$ for given $\y^{(j)}$ according to Theorem \ref{thm:supp} in case \ref{item:thm1_bound_J} (left) and case \ref{item:thm1_b} (right).}
\label{fig2}
\end{figure}
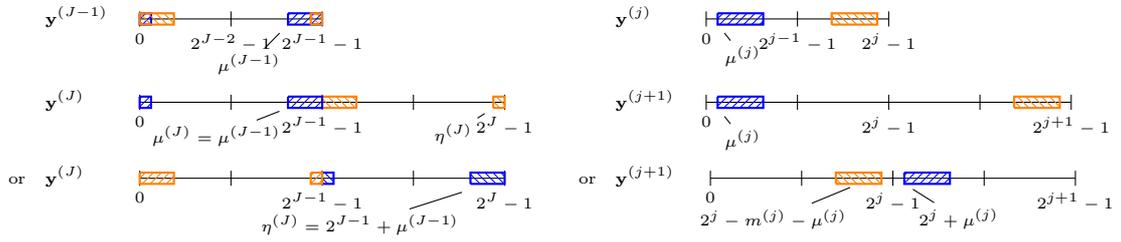 
\normalsize

Let us now consider the support properties of the periodization $\y^{(j)}\in\RR^{2^j}$ for $j<J$. 
\begin{lem}\label{lem:support}
Let $N=2^{J-1}$ and $\x\in\RR^N$ have a one-block support of length $m<N$. Set $\y=(\x^T,(\J_N\x)^T)^T$ and assume that $\y$ satisfies $(\ref{eq:suppose})$. For $j\in\{0,\dotsc,J-1\}$ let $\y\jj$ be the $2^j$-length periodization of $\y$ according to $(\ref{eq:per_sum})$. Then $\y^{(j)}$ possesses either 
\begin{enumerate}[label=\upshape\Alph*),ref=\Alph*]
\item\label{item:supp_a} the (reflected) one-block support $S\jj=I_{\mu^{(j)}, \nu^{(j)}}$ of length $m^{(j)} \le 2m$ for some $\mu\jj\in\{0,\dotsc,2^j-1\}$ with $\nu\jj=2^j-1-\mu\jj$, which is centered around the middle of the vector, i.e., $2^{j-1}-1$ and $2^{j-1}$, or around the boundary, i.e., $2^j-1$ and $0$. Here, $m\jj=\nu\jj-\mu\jj+1$ if $\nu\jj\geq\mu\jj$ and $m\jj=2^j-\mu\jj+\nu\jj+1$ if $\nu\jj<\mu\jj$, or 
\item\label{item:supp_b} the (reflected) two-block support $S\jj=I_{\mu^{(j)}, \nu^{(j)}}\cup I_{2^{j}-1-\nu^{(j)}, 2^{j}-1-\mu^{(j)}}$ for  $\mu\jj,\nu\jj\in\{0,\dotsc,2^j-1\}$, where the two blocks  have length $n\jj=m$ and are separated.
\end{enumerate}
\end{lem}
\begin{proof} 
By definition of the periodization $\y^{(j)}$ in (\ref{eq:per_sum}) the number of support indices $k$ with $y_{k}^{(j)} \neq 0$  can never exceed the number of nonzero elements of $\y^{(j+1)}$, also called sparsity. If $\y^{(j+1)}$ has a two-block support, then $\y^{(j)}$ possesses at most two support blocks, as the two blocks of $\y^{(j+1)}$ are either mapped to two blocks of the same length in $\y^{(j)}$ or to one block consisting of the two partially overlapping blocks of $\y^{(j+1)}$ by (\ref{eq:per_el}). Thus, if $\y^{(j)}$ has a one-block support, its length $m\jj$ is at most $2m$. Moreover, by the symmetry property (\ref{eq:symmetry}), the block has to be centered around the middle of the vector or around its boundary. It also follows from (\ref{eq:symmetry}) that the blocks are reflections of each other in the two-block case. If we have obtained $\y^{(l)}$ with a one-block support, all shorter periodizations $\y^{(j)}$, $j\in\{0, \ldots , l-1\}$, possess a one-block support of length $m\jj\leq2m$ as well. Figure \ref{fig1} and Figure \ref{fig2} (right) depict cases \ref{item:supp_a} and \ref{item:supp_b}, respectively.
\end{proof}
We always denote by $\mu\jj$ and $\nu\jj$ the \emph{first and last support index} of $\y\jj$ if $\y\jj$ has a one-block support, or the \emph{first index of the first support block} and the \emph{last index of the first support block} if $\y\jj$ has a two-block support. In the one-block case the \emph{support length} of $\y\jj$ is $m\jj$, and in the two-block case we denote the \emph{length of the two support blocks} by $n\jj$ for algorithmic purposes, even though $n\jj=m$ for exact data.

If $\y^{(j)}$ has a one-block support of length $2^{j-1} < m^{(j)} \le 2^{j}$, the first support index may not be uniquely determined. For example, for $\y^{(3)} = (0,1,1,0,0,1,1,0)^{T}$, which by definition can be considered to have a one-block support of length $6$, the support interval can be either $S^{(3)}=I_{1,6}$ or $S^{(3)}=I_{5,2}$. Note that by Lemma \ref{lem:support} the support of $\y\jj$ is symmetric, i.e., $S^{(j)} = I_{\mu^{(j)},2^{j}-1-\mu^{(j)}}$, which can be used to exclude some possible first support indices. If the  first support index is still not unique, we choose $0 \le \mu^{(j)} < 2^{j-1}-1$ such that $S^{(j)}$ is centered  around the middle of the vector, i.e., around $2^{j-1}-1$ and $2^{j-1}$. In the example above we choose $\mu^{(3)} \coloneqq 1$. If $m=2^{j}$, we just fix $\mu^{(j)}\coloneqq 0$. 
\begin{ex}
1. Let $\x=(0,x_{1},x_{2},0,0,0,0,0)^{T} \in {\mathbb R}^{8}$ with nonzero entries $x_{1}$, $x_{2}$, i.e., with one-block support $S^\x=I_{1,2}$ of length $m=2$. Then $\y$ and its periodizations are
\begin{align*}
 \y= \y^{(4)} & = (0,x_{1}, x_{2}, 0,0,0,0,0,0,0,0,0,0,x_{2},x_{1},0)^{T}, \\
 \y^{(3)} & = (0,x_{1}, x_{2},0,0,x_{2},x_{1},0)^{T}, \\
 \y^{(2)} & = (0,x_{1}+ x_{2},x_{1}+x_{2},0)^{T},\\
 \y^{(1)} & = (x_{1}+ x_{2}, x_{1}+ x_{2})^{T}, \\
 \y^{(0)} & = \left(2(x_{1}+x_{2})\right)^{T}.
 \end{align*}
 Here, $\y$ has the reflected block support $S^{(4)}=I_{1,2} \cup I_{13,14}$ of length $n^{(4)}=m=2$, as in case \ref{item:supporty_i} of Lemma \ref{lem:supporty}, and $\y^{(3)}$ has the two-block support $S^{(3)}=I_{1,2}\cup I_{5,6}$ of length $n^{(3)}=m=2$. The periodization $\y^{(2)}$ has the one-block support $S^{(2)}=I_{1,2}$ of length $m^{(2)}=2 < 2m$, centered around the middle of the vector, i.e., $1$ and $2$, with $\mu^{(2)} = 1$. The vectors $\y^{(1)}$ and $\y^{(0)}$ both have full support, which can be interpreted as a one-block support centered around the middle with $\mu^{(1)}=\mu^{(0)} = 0$, $m^{(1)}=2$ and $m^{(0)}=1$.
 
2. Let $\x = (x_{0}, x_{1}, 0,0,0,0,0,x_{7})^{T}$ with nonzero entries $x_0$, $x_{1}$ and $x_7$, i.e., with one-block support  $S^{\x}=I_{7,1}$ of length $m=3$. Then we obtain 
\begin{align*}
 \y= \y^{(4)} & = (x_{0},x_{1}, 0, 0,0,0,0,x_{7},x_{7},0,0,0,0,0,x_{1},x_{0})^{T}, \\
 \y^{(3)} & = (x_{0}+x_{7},x_{1},0,0,0,0,x_{1},x_{0}+x_{7})^{T}, \\
 \y^{(2)} & = (x_{0} + x_{7},x_{1},x_{1},x_{0}+x_{7})^{T},\\
 \y^{(1)} & = (x_{0}+ x_{1}+ x_{7}, x_{0}+x_{1}+ x_{7})^{T}, \\
 \y^{(0)} & = \left(2(x_{0}+x_{1}+x_{7})\right)^{T}.
 \end{align*}
 Here, $\y$ has the reflected block support $S^{(4)}=I_{7,8} \cup I_{14,1}$ with block lengths $2$ and $4$, as in case \ref{item:supporty_i} of Lemma \ref{lem:supporty}. The vector $\y^{(3)}$ has the one-block support $S^{(3)}=I_{14,1}$ of length $m^{(3)}=4$ with $\mu^{(3)}=6$, centered around $15$ and $0$. All shorter periodizations have a one-block support as well, but centered around the middle  with $\mu^{(2)}=\mu^{(1)} = \mu^{(0)} = 0$. 
 \end{ex}
In our algorithm we want to reconstruct  $\y^{(j+1)}$ iteratively for $j\in\{0, \ldots ,J-1\}$, using in each step that $\y\jj$ is already known. In the following theorem we summarize the observations about the support properties of $\y^{(j+1)}$ given the support of $\y^{(j)}$.
\begin{thm} \label{thm:supp}
Let $N=2^{J-1}$ and $\x\in\RR^N$ have a one-block support of length $m<N$. Set $\y=(\x^T,(\J_N\x)^T)^T$ and assume that $\y$ satisfies $(\ref{eq:suppose})$. For $j\in\{0,\dotsc,J-1\}$ let $\y\jj$ be the $2^j$-length periodization of $\y$ according to $(\ref{eq:per_sum})$.
\begin{enumerate}
\labitem{$\mathrm{A}_1$}{item:thm1_middle} If $\y^{(j)}$ possesses the one-block support $S\jj=I_{\mu^{(j)},2^{j}-1-\mu^{(j)}}$ of length $m\jj\leq2m$ with $m\jj<2^j$ centered around the middle of the vector, i.e., around $2^{j-1}-1$ and $2^{j-1}$, then $\y^{(j+1)}$ possesses the two-block support  
\begin{equation*}
S\jjj=I_{\mu^{(j+1)},\nu^{(j+1)}} \cup I_{2^{j+1}-1-\nu^{(j+1)},2^{j+1}-1-\mu^{(j+1)}},
\end{equation*}
with two  blocks of length $n\jjj=m$, where either $\mu^{(j+1)} = \mu^{(j)}$ or $2^{j+1}-1-\nu^{(j+1)}= \mu^{(j)} + 2^{j}$, see Figure $\ref{fig1}$ (left).

\labitem{$\mathrm{A}_2$}{item:thm1_full}
If $\y\jj$ possesses the one-block support $S\jj=I_{0,2^j-1}$ of length $m\jj=2^j$ and $j<J-1$, then $\y\jjj$ has a one-block support of length $m^{(j+1)}\geq m\jj$. In the special case $j=J-1$, $\y=\y^{(J)}$ has a two-block support with two blocks of possibly different lengths or a one-block support of length $m^{(J)}\geq m^{(J-1)}$.

\labitem{$\mathrm{A}_3$}{item:thm1_bound} If $j < J-1$ and $\y^{(j)}$ possesses the one-block support $I_{\mu^{(j)},2^{j}-1-\mu^{(j)}}^{(j)}$ of length $m^{(j)}<2^j$ centered around the boundary of the vector, i.e., around $2^j-1$ and $0$, then $\y^{(j+1)}$ possesses the one-block support
\begin{equation*}
S\jjj=I_{\mu^{(j+1)}, 2^{j+1} -1-\mu^{(j+1)}}
\end{equation*} 
 of length $m^{(j+1)}=m^{(j)}$, centered around the middle or the boundary of the vector, i.e., around $2^j-1$ and $2^j$ or around $2^{j+1}-1$ and $0$. Here, $\mu^{(j+1)} = \mu^{(j)}$ or $\mu^{(j+1)} = \mu^{(j)} + 2^{j}$, see Figure $\ref{fig1}$ (right). 

\labitem{$\mathrm{A}_4$}{item:thm1_bound_J} If $j=J-1$ and $\y^{(J-1)}$ has the one-block support $S^{(J-1)}=I_{\mu^{(J-1)},2^{J-1}-1-\mu^{(J-1)}}$ with $m^{(J-1)}<2^{J-1}$ centered around the boundary of the vector, i.e., around $2^{J-1}-1$ and $0$, then $\y= \y^{(J)}$ possesses the two-block support 
\begin{equation}\label{eq:thm1_bound_J} 
 S^{(J)}=I_{\mu^{(J)},2^{J}-1-\mu^{(J)}} \cup I_{\eta^{(J)},2^{J}-1-\eta^{(J)}}\qquad  \text{with} \quad \mu^{(J)}<2^{J-1} \leq \eta^{(J)},
\end{equation}
where the two blocks may have different lengths. In the boundary cases $\mu^\x=0$ or $\nu^\x=2^{J-1}-1$ one of these blocks is empty. We always have that either $\mu^{(J)} = \mu^{(J-1)}$  or $\eta^{(J)}= 2^{J-1}+\mu^{(J-1)}$, but $\mu^{(J)}$ and $\eta^{(J)}$ are not both known, see Figure $\ref{fig2}$ (left). 

\labitem{\textnormal{B}}{item:thm1_b} If $\y^{(j)}$ possesses the two-block support $S\jj=I_{\mu^{(j)},\nu^{(j)}} \cup I_{2^{j}-1-\nu^{(j)},2^{j}-1-\mu^{(j)}}$ with block length $n\jj=m$, then $\y^{(j+1)}$ has the two-block support  
\begin{equation}\label{eq:thm1_b}
S\jjj=I_{\mu^{(j+1)},\nu^{(j+1)}} \cup I_{2^{j+1}-1-\nu^{(j+1)},2^{j+1}-1-\mu^{(j+1)}}
\end{equation} 
with block length $n\jjj=m$, where $\mu^{(j+1)} = \mu^{(j)}$ or $\mu^{(j+1)} = 2^{j}-m- \mu^{(j)}$, see Figure $\ref{fig2}$ (right). If $\mu\jjj=\mu\jj$, the first blocks of $\y\jj$ and $\y\jjj$ are identical, with the same support, and the second block of $\y\jjj$ is the second block of $\y\jj$, shifted by $2^j$. Otherwise the first block of $\y\jjj$ is the second block of $\y\jj$, with the same support, and the second block of $\y\jjj$ is the first block of $\y\jj$, shifted by $2^j$.
\end{enumerate}
\end{thm}

\begin{proof}
Cases \ref{item:thm1_middle} to \ref{item:thm1_bound_J}, henceforth subsumed to Case \labitems{A}{item:thm1_a}, summarize the support properties of $\y^{(j+1)}$ if $\y^{(j)}$ possesses  a one-block support. Assertion \labitems{B}{item:thm1_b_bold}\ covers the case that $\y^{(j)}$ has a two-block support. 

All observations about the possible support blocks of $\y^{(j+1)}$  follow immediately from the known support $S\jj$ of $\y^{(j)}$ and (\ref{eq:per_sum}), using the results from Lemma \ref{lem:supporty} and Lemma \ref{lem:support}. Figures \ref{fig1}  and \ref{fig2} illustrate the described cases.
\end{proof}

\section{Iterative Sparse DFT Procedures}\label{sec:dft_procedures}
There is an important difference between  cases \ref{item:thm1_middle} to \ref{item:thm1_bound_J} of Theorem \ref{thm:supp} on the one hand and case \ref{item:thm1_b} on the other hand.
In cases \ref{item:thm1_middle} to \ref{item:thm1_bound_J} $\y\jj$ has a one-block support that usually contains overlapping entries of the original vector $\y$, i.e., some entries of $\y^{(j)}$ are obtained as sums of nonzero entries of $\y$. In case \ref{item:thm1_b}, however, both support blocks of $\y\jj$ are of length $n\jj=m$ and they are separated. Thus, the nonzero entries of $\y\jj$ and $\y\jjj$ are the same and we only have to find the first support indices of the blocks in $\y\jjj$ to obtain it from $\y\jj$.
In the following two sections  we therefore derive two different strategies for computing $\y^{(j+1)}$; the first one has to be employed in case \ref{item:thm1_a}, and the second one in case \ref{item:thm1_b}.

\subsection{A DFT Procedure for Case \ref{item:thm1_a}: One-block Support}
We assume that $\y^{(j)}$ possesses the one-block support 
\begin{equation}\label{eq:blocksupp}
S\jj=I_{\mu^{(j)},2^{j}-1-\mu^{(j)}} = I_{\mu^{(j)}, \left(\mu^{(j)}+m^{(j)}-1\right) \mods  2^{j}} 
\end{equation}
of length $m^{(j)} \le 2m$. Then it follows from (\ref{eq:per_el}) and Theorem \ref{thm:supp}, case \ref{item:thm1_a} that the support $S^{(j+1)}$  of $\y^{(j+1)}$ satisfies 
\[
 S^{(j+1)} \subset I_{\mu^{(j)}, \mu^{(j)}+m^{(j)}-1} \cup I_{2^j+\mu^{(j)}, \left(2^j+\mu^{(j)}+m^{(j)}-1\right) \mods 2^{j+1}}.
\]
The procedure developed hereafter utilizes that $\y\jjj$ is symmetric and thus determined by its first half. We define the two partial vectors 
\[
 \y^{(j+1)}_{(0)}\coloneqq\left(y^{(j+1)}_k\right)_{k=0}^{2^j-1} \quad \text{and}\quad  \y^{(j+1)}_{(1)}\coloneqq\left(y^{(j+1)}_k\right)_{k=2^j}^{2^{j+1}-1}
\]
of $\y^{(j+1)}$. Both halves of $\y\jjj$ have a one-block support, as can be seen in Figure \ref{fig1} and Figure \ref{fig2} (left). Hence, also the following inclusion holds for the support of $\y\jjj$, 
\begin{align}
 S\jjj\subset&\left\{\left(\mu\jj+r\right)\mods 2^j:r\in\left\{0,\dotsc,m\jj-1\right\}\right\} \notag \\
 \cup&\left\{2^j+\left(\mu\jj+r\right)\mods 2^j:r\in\left\{0,\dotsc,m\jj-1\right\}\right\}, \label{eq:supp_incl}
\end{align}
where the support of $\y\jjj_{(0)}$ is contained in the first set and the support of $\y\jjj_{(1)}$ is contained in the second set. In particular, $\y\jjj$ has at most $2m^{(j)}$ nonzero entries. The periodization property (\ref{eq:per_el}) implies that 
\begin{equation}\label{eq:periodization}
 \y^{(j)}=\y^{(j+1)}_{(0)}+\y^{(j+1)}_{(1)}, 
\end{equation}
i.e., $\y^{(j+1)}$ is already determined by $\y^{(j)}$ and $\y^{(j+1)}_{0}$.
Moreover, we can use the symmetry property (\ref{eq:sym}) to find $\y\jjj_{(1)}=\J_{2^j}\y\jjj_{(0)}$ via a permutation instead of solving (\ref{eq:periodization}).

Since the support of $\y_{0}^{(j+1)}$ is a subset of the support of $\y\jj$ by (\ref{eq:supp_incl}), we have to recover at most $m\jj$  nonzero entries of $\y_{0}^{(j+1)}$. In order to efficiently compute these entries we consider restrictions of $\y\jj$ and $\y\jjj_{(0)}$ to vectors of length $2^{L\jj}$, where $2^{L^{(j)}-1} < m^{(j)} \le 2^{L^{(j)}}$, taking into account all nonzero entries.  We then show that $\y^{(j+1)}_{0}$ and thus $\y^{(j+1)}$ can be computed using essentially one DFT of length $2^{L^{(j)}}$ and some further operations of complexity ${\mathcal O}\left(m^{(j)}\right)$.
For this purpose we need to employ the known vector $\y^{(j)}$ from the previous iteration step and $2^{L^{(j)}}$ suitably chosen oddly indexed entries of $\widehat{\y\jjj}$, which can be obtained from $\widehat{\y}$ by Lemma \ref{lem:periodization}.

The efficient computation of $\y^{(j+1)}$ will be based on the following theorem.
\begin{thm}\label{thm:oneblock}
Let $N=2^{J-1}$ and $\x\in\RR^N$ have a one-block support of length $m<N$. Set $\y=(\x^T,(\J_N\x)^T)^T$ and assume that $\y$ satisfies $(\ref{eq:suppose})$. For $j\in\{0,\dotsc,J-1\}$ let $\y\jj$ be the $2^j$-length periodization of $\y$ according to $(\ref{eq:per_sum})$ and suppose that $\y\jj$ has the one-block support $S\jj$ of length $m\jj$ as in $(\ref{eq:blocksupp})$. Assume that we have access to all entries of $\widehat{\y} = \left(\widehat{y}_{k}\right)_{k=0}^{2^{J}-1}$.
Further, let $L^{(j)}\coloneqq\left\lceil\log_2m^{(j)}\right\rceil\leq j$ and define the restrictions of $\y^{(j)}$ and $\y^{(j+1)}_{(0)}$ to $2^{L\jj}$-length vectors,
\[
  \z\jj \coloneqq\left(y\jj_{\left(\mu\jj+r\right) \mods 2^j}\right)_{r=0}^{2^{L\jj}-1} \quad \text{and}\quad 
  \z\jjj_{(0)} \coloneqq\left(y\jjj_{\left(\mu\jj+r\right)\mods 2^j}\right)_{r=0}^{2^{L\jj}-1} .
\]
Then
\begin{equation}\label{z0}
  \z\jjj_{(0)}=\frac{1}{2}\left({\W_{(1)}\jj}^{-1}\cdot\F_{2^{L\jj}}^{-1}\cdot{\W_{(0)}\jj}^{-1}\left(\widehat{y}_{2^{J-L\jj}p+2^{J-j-1}}\right)_{p=0}^{2^{L\jj}-1}+\z\jj\right),
\end{equation}
where ${\W_{(0)}\jj}\coloneqq\diag\left( \omega_{2^{L^{(j)}}}^{p \mu^{(j)}}\right)_{p=0}^{2^{L^{(j)}}-1}$ and ${\W_{(1)}\jj} \coloneqq\diag \left( \omega_{2^{j+1}}^{\left(\mu\jj+r\right) \mods  2^{j}}\right)_{r=0}^{2^{L^{(j)}}-1}$, 
and the periodization $\y\jjj$ is completely determined by
\begin{align} \label{y0}
  \left(y\jjj_{(0)}\right)_{\left(\mu\jj+k\right)\mods 2^j}&=\begin{cases}
                                                            \left(z\jjj_{(0)}\right)_k, & k\in\left\{0,\dotsc,2^{L\jj}-1\right\}, \\
                                                            0, & \text{else,}
                                                           \end{cases} \\
  \left(y\jjj_{(1)}\right)_{2^j-1-\left(\mu\jj+k\right)\mods 2^j}&=\begin{cases}
                                                            \left(z\jjj_{(0)}\right)_k, & k\in\left\{0,\dotsc,2^{L\jj}-1\right\}, \\
                                                            0, & \text{else.}
                                                           \end{cases} 
                                                           \label{y1}
 \end{align}
\end{thm}

\begin{proof}
It suffices to consider the oddly indexed entries $\widehat{y\jjj}_{2k+1}$ of $\widehat{\y\jjj}$ for $k\in\left\{0, \ldots , 2^{j}-1\right\}$. With (\ref{eq:periodization}) we obtain
\begin{align*}
 \widehat{y\jjj}_{2k+1}=&\left(\left(\omega_{2^{j+1}}^{(2k+1)l}\right)_{l=0}^{2^{j+1}-1}\right)^{T} \begin{pmatrix}
                                                                              \y^{(j+1)}_{(0)} \\
                                                                              \y^{(j+1)}_{(1)}
                                                                             \end{pmatrix} \\
 =&\left(\left(\omega_{2^{j+1}}^{(2k+1)l}\right)_{l=0}^{2^j-1} \right)^{T}\y^{(j+1)}_{(0)}
 +\left( \left(\omega_{2^{j+1}}^{(2k+1)l}\right)_{l=2^j}^{2^{j+1}-1} \right)^{T}\left(\y^{(j)}-\y^{(j+1)}_{(0)}\right) \\
 =&\left( \left(\omega_{2^{j+1}}^{(2k+1)l}\right)_{l=0}^{2^j-1} \right)^{T}\y^{(j+1)}_{(0)}- \left( \left(\omega_{2^{j+1}}^{(2k+1)l}\right)_{l=0}^{2^j-1} \right)^{T} \left(\y^{(j)}-\y^{(j+1)}_{(0)}\right) \\
 =& \left( \left(\omega_{2^{j+1}}^{(2k+1)l}\right)_{l=0}^{2^j-1} \right)^{T} \left(2\y^{(j+1)}_{(0)}-\y^{(j)}\right). 
\end{align*}
Using Lemma \ref{lem:periodization} we find that
\begin{equation}\label{eq:eqsys1}
 \left(\widehat{y}_{2^{J-j-1}(2k+1)}\right)_{k=0}^{2^j-1}  
 =\left(\widehat{y^{(j+1)}}_{2k+1}\right)_{k=0}^{2^j-1}
 =\left(\omega_{2^{j+1}}^{(2k+1)l}\right)_{k,\,l=0}^{2^j-1}\left(2\y^{(j+1)}_{(0)}-\y^{(j)}\right), 
\end{equation}
so $\y^{(j+1)}_{(0)}$ can be computed from $\y^{(j)}$ and the oddly indexed entries of $\widehat{\y\jjj}$. 

By definition of $L\jj$ we have that $2^{L^{(j)}-1}<m^{(j)}\leq2^{L\jj}$, and the $m\jj$ nonzero entries of $\y\jj$ are taken into account by the $2^{L\jj}$-length restriction $\z\jj$ of $\y\jj$. Similarly, by (\ref{eq:per_el}), $\z\jjj_{(0)}$ takes into account the at most $m\jj$ nonzero entries of $\y^{(j+1)}_{(0)}$. 
We can therefore restrict (\ref{eq:eqsys1}) to the vectors $\z\jj$ and $\z\jjj_{(0)}$, which yields 
\begin{equation}
 \left(\widehat{y}_{2^{J-j-1}(2k+1)}\right)_{k=0}^{2^j-1} 
  =\left(\omega_{2^{j+1}}^{(2k+1)\left(\left(\mu\jj+r\right)\mods 2^j\right)}\right)_{k,\,r=0}^{2^j-1,\,2^{L\jj}-1}\left(2\z\jjj_{(0)}-\z\jj\right). \label{eq:res_eqsys}
\end{equation}  
As $\z\jjj_{(0)}$ and $\z\jj$ have length $2^{L\jj}$, it suffices to consider the $2^{L\jj}$ equations corresponding to $k\coloneqq2^{j-L\jj} p$ for $p\in\left\{0,\dotsc,2^{L\jj}-1\right\}$ in (\ref{eq:res_eqsys}). We obtain the factorization
  \begin{align}
  &\left(\widehat{y}_{2^{J-j-1}\left(2^{j+1-L\jj}p+1\right)}\right)_{p=0}^{2^{L\jj}-1} \notag \\
  =&\left(\omega_{2^{j+1}}^{\left(2^{j+1-L\jj}p+1\right)\left(\left(\mu\jj+r\right)\mods2^j\right)}\right)_{p,\,r=0}^{2^{L\jj}-1}
  \left(2\z\jjj_{(0)}-\z\jj\right) \notag \\
  =&\diag\left(\omega_{2^{L\jj}}^{p\mu\jj}\right)_{p=0}^{2^{L\jj}-1}\cdot 
  \left(\omega_{2^{L\jj}}^{pr}\right)_{p,\,r=0}^{2^{L\jj}-1} 
  \cdot\diag\left(\omega_{2^{j+1}}^{\left(\mu\jj+r\right)\mods2^j}\right)_{r=0}^{2^{L\jj}-1}
  \left(2\z\jjj_{(0)}-\z\jj\right) \notag \\
  =& \, \W_{(0)}\jj\cdot\F_{2^{L\jj}}\cdot\W_{(1)}\jj\left(2\z\jjj_{(0)}-\z\jj\right). \label{eq:final_eqsys}
 \end{align}
Since all matrices occurring in (\ref{eq:final_eqsys}) are invertible, we derive $\z\jjj_{(0)}$ as in (\ref{z0}).
Then  $\y\jjj$ is given as in (\ref{y0}) and (\ref{y1}) by definition of $\z\jjj_{(0)}$ and symmetry (\ref{eq:sym}). Note that if $L\jj=2^j$, then $\z\jj=\y\jj$ and $\z\jjj_{(0)}=\y\jjj_{(0)}$.
\end{proof}

\subsection{A DFT Procedure for Case \ref{item:thm1_b_bold}: Two-block Support} 
We still have to devise a procedure for reconstructing $\y\jjj$ from $\y\jj$ in Case \ref{item:thm1_b} of Theorem \ref{thm:supp}, i.e., if $\y\jj$ has a two-block support of the form  
\begin{equation} \label{eq:twosupp} 
 S\jj=I_{\mu\jj, \nu\jj} \cup I_{2^{j}-1-\nu\jj, 2^{j}-1-\mu\jj}=I_{\mu\jj, \mu\jj+m-1} \cup I_{2^{j}-m-\mu^{(j)}, 2^{j}-1-\mu^{(j)}}
\end{equation}
with two blocks of length $n\jj=m$, and $\nu^{(j)} =\mu^{(j)}+m-1$.
We recall that by Theorem \ref{thm:supp}, Case \ref{item:thm1_b}, the length $n\jj$ of both blocks in $\y\jj$ is the same as the length $n\jjj$ of both blocks in $\y^{(j+1)}$ and also the same as the support length $m$ of $\x$.
Furthermore,  all entries of $\y^{(j+1)}$ are already determined; we just have to find out whether the first support block of $\y\jj$ remains at the same position in $\y\jjj$ or whether its support is shifted by $2^j$. The other support block is obtained as the reflection of this block, according to the symmetry property (\ref{eq:sym}), see Figure \ref{fig2} (right). In order to decide which of these possibilities for $\y^{(j+1)}$ is true, we employ one oddly indexed nonzero entry of $\widehat{\y\jjj}$. In a first step we show that such a nonzero entry can be found efficiently.

\begin{lem}\label{lem:nonzero}
Let $N=2^{J-1}$ and $\x\in\RR^N$ have a one-block support of length $m<N$. Set $\y=(\x^T,(\J_N\x)^T)^T$ and assume that $\y$ satisfies $(\ref{eq:suppose})$. For $j\in\{0,\dotsc,J-1\}$ let $\y\jj$ be the $2^j$-length periodization of $\y$ according to $(\ref{eq:per_sum})$ and suppose that $\y\jj$ has the two-block support $S\jj$ as in $(\ref{eq:twosupp})$. Assume that we have access to all entries of $\widehat{\y} = \left(\widehat{y}_{k}\right)_{k=0}^{2^{J}-1}$. Then $\left(\widehat{y\jjj}_{2k+1}\right)_{k=0}^{2m-1}$ has at least one nonzero entry.
\end{lem}

\begin{proof}  
Theorem \ref{thm:supp}, case \ref{item:thm1_b} yields that $\y\jjj$ has a two-block support as in (\ref{eq:thm1_b}), which we denote by $S\jjj$. Considering the first $2m$ oddly indexed entries of $\widehat{\y\jjj}$, we find
\begin{align}
\left(\widehat{y^{(j+1)}}_{2k+1}\right)_{k=0}^{2m-1} 
=& \left( \sum_{l \in S^{(j+1)}} \omega_{2^{j+1}}^{(2k+1)l}y_{l}^{(j+1)} \right)_{k=0}^{2m-1} \notag \\
=& \left( \omega_{2^{j+1}}^{(2k+1)l} \right)_{k=0,\,l \in S^{(j+1)}}^{2m-1}\left( y_{l}^{(j+1)} \right)_{l \in S^{(j+1)}} \notag \\
=& \left( \left(\omega_{2^j}^l\right)^k\right)_{k=0,\, l\in S\jjj}^{2m-1} \diag \left(\left( \omega_{2^{j+1}}^l \right)_{l \in S\jjj}\right)  \left( y_l\jjj \right)_{l \in S\jjj}.\label{eq:proof_proc_b}
\end{align}
Assume that $\left(\widehat{y\jjj}_{2k+1}\right)_{k=0}^{2m-1}=\boldsymbol{0}_{2m}$. The Vandermonde matrix $\left(\left(\omega_{2^j}^l\right)^k \right)_{k=0,\, l \in S\jjj}^{2m-1}$ is invertible if and only if the $\omega_{2^j}^l$ are pairwise distinct for all $l \in S\jjj$, or equivalently, if the $l \mods 2^j$ are pairwise distinct for all $l \in S\jjj$. Since $\y\jj$ already has a two-block support with separated blocks, this holds true. Hence, as $\left( y_l^{(j+1)} \right)_{l \in S^{(j+1)}} \neq {\mathbf 0}_{2m}$ by definition of $\y$ and the matrices occurring in (\ref{eq:proof_proc_b}) are invertible, we obtain a contradiction. For the implementation of this procedure set
\[
 k_0\coloneqq\underset{k\in\{0,\dotsc,2m-1\}}{\argmax}\left\{\left|\widehat{y}_{2^{J-j-1}(2k+1)}\right|\right\}.
\]
Then $\widehat{y\jjj}_{2k_0+1}\neq0$ and it is likely that this entry is not too close to zero, which is supported empirically by the numerical experiments in Section \ref{sec:numerics}.
\end{proof} 
Now we can show how the support of $\y\jjj$ can be determined from $\y\jj$ and one oddly indexed nonzero entry of $\widehat{\y\jjj}$.
\begin{thm}\label{thm:twoblocks}
Let $N=2^{J-1}$ and $\x\in\RR^N$ have a one-block support of length $m<N$. Set $\y=(\x^T,(\J_N\x)^T)^T$ and assume that $\y$ satisfies $(\ref{eq:suppose})$. For $j\in\{0,\dotsc,J-1\}$ let $\y\jj$ be the $2^j$-length periodization of $\y$ according to $(\ref{eq:per_sum})$ and suppose that $\y\jj$ has the two-block support $S\jj$ as in $(\ref{eq:twosupp})$. Assume that we have access to all entries of $\widehat{\y} = \left(\widehat{y}_{k}\right)_{k=0}^{2^{J}-1}$. Then $\y\jjj$ can be uniquely recovered from $\y\jj$ and one nonzero entry of $\left(\widehat{y}_{2^{J-j-1}(2k+1)} \right)_{k=0}^{2m-1}$.
\end{thm}
\begin{proof}
If $\y^{(j)}$ is known and has a two-block support, there are two possibilities for the periodized vector $\y\jjj$, one of which is obtained by shifting the other by $2^j$ by Theorem \ref{thm:supp}, case \ref{item:thm1_b}, see Figure \ref{fig2} (right). We denote these two vectors by $\uu^0= \left(u_{k}^{0}\right)_{k=0}^{2^{j+1}-1}$ and $\uu^1= \left(u_{k}^{1}\right)_{k=0}^{2^{j+1}-1}$, and obtain that
\begin{align*}
  u^0_k & \coloneqq\begin{cases}
           y\jj_k, & k\in\left\{\mu\jj,\dotsc,\mu\jj+m-1\right\}, \\
           y\jj_{k-2^j}, & k\in\left\{2^{j+1}-m-\mu\jj,\dotsc,2^{j+1}-1-\mu\jj\right\}, \\
           0, & \text{else}
          \end{cases} \qquad \text{and} \\
  u^1_{k} & \coloneqq u^0_{\left(2^j+k\right)\mods 2^{j+1}}, \quad k\in\{0,\dotsc,2^{j+1}-1\}.
\end{align*}  
Lemma \ref{lem:shift} implies that
 \begin{equation}\label{eq:comparison}
  \widehat{u^1}_{2k+1}= - \widehat{u^0}_{2k+1}, \qquad k\in\left\{0,\dotsc,2^j-1\right\},
 \end{equation}
 for all oddly indexed entries of $\widehat{u^0}$ and $\widehat{u^1}$. In order to decide whether $\y^{(j+1)} = \uu^{0}$ or $\y^{(j+1)} = \uu^{1}$ we compare a nonzero entry $\widehat{y\jjj}_{2k_0+1} = \widehat{y}_{2^{J-j-1}(2k_0+1)} \neq 0$ with the corresponding entry of  $\widehat{\uu^{0}}$. It follows from Lemma \ref{lem:nonzero} that $\widehat{y\jjj}_{2k_0+1}$ can be found by examining $2m$ entries. If $\widehat{u^0}_{2k_0+1} = \widehat{y\jjj}_{2k_0+1}$, we conclude that $\y^{(j+1)} = \uu^{0}$, and if $\widehat{u^0}_{2 k_0+1} = - \widehat{y\jjj}_{2k_0+1}$, then $\y^{(j+1)} = \uu^{1}$ by (\ref{eq:comparison}). Numerically, we set $\y^{(j+1)} = \uu^{0}$ if 
 \[
 \left| \widehat{u^0}_{2k_0+1} - \widehat{y\jjj}_{2k_0+1}\right| < \left| \widehat{u^0}_{2k_0+1} + \widehat{y\jjj}_{2k_0+1}\right|
 \]
and $\y^{(j+1)} = \uu^{1}$ otherwise. The required entry of $\widehat{\uu^0}$ can be computed from $\y\jj$ using $\mathcal{O}(m)$ operations,
\begin{align*}
  \widehat{u^0}_{2k_0+1}=&\sum_{l=0}^{2^{j+1}-1} \omega_{2^{j+1}}^{(2k_0+1)l}u^0_l
  =\sum_{l=\mu\jj}^{\mu\jj+m-1}\omega_{2^{j+1}}^{(2k_0+1)l} y\jj_l 
  +\sum_{l=2^{j+1}-m-\mu\jj}^{2^{j+1}-1-\mu\jj}\omega_{2^{j+1}}^{(2k_0+1)l}y\jj_{l-2^j}.
 \end{align*} 
\end{proof}

\section{The Sparse FFT and the Sparse Fast DCT}
In Section \ref{sec:dft_procedures} we have presented all procedures necessary to derive the new sparse FFT for vectors $\y \in {\mathbb R}^{2N}$ that have a reflected block support and satisfy (\ref{eq:suppose}). Using Lemma \ref{lem:dctx_from_yhat} we also obtain a new sparse fast DCT algorithm for vectors with one-block support. Note that neither of the procedures for reconstructing $\y\jjj$ from $\y\jj$ and $\widehat{\y}$ introduced above requires a priori knowledge of the length of the blocks in $\y$.
\subsection{A Sparse FFT for Vectors with Reflected Block Support}\label{sec:dft_alg}
Let us assume that $N=2^{J-1}$ and $\y \in {\mathbb R}^{2N}$ has a reflected block support of unknown block length $m< N$ and satisfies (\ref{eq:suppose}), i.e., there is no cancellation of nonzero entries in any of the periodization steps.
We suppose that we have  access to all entries of $\widehat{\y} \in {\mathbb C}^{2N}$.

The algorithm starts with the initial vector $\y^{(0)}=\sum_{l=0}^{2N-1}y_l=\widehat{y}_0$, which has a one-block support. For $j\in\{0, \ldots J-1\}$ we perform the following steps.
\begin{enumerate}
\labitem{1a}{item:alg_1a} If $\y^{(j)}$ possesses a one-block support, apply the DFT procedure given in Theorem \ref{thm:oneblock} to recover $\y^{(j+1)}$.
\labitem{1b}{item:alg_1b} If $\y^{(j)}$ possesses a two-block support, apply the DFT procedure given in Theorem \ref{thm:twoblocks} to recover $\y^{(j+1)}$.
\labitem{2}{item:alg_2} Detect the support structure of $\y\jjj$.
\end{enumerate}
A stable method to detect the support structure of $\y\jjj$ using ${\mathcal O}\left(m^{(j)}\right)$ operations is given in Section \ref{sec:detection}. Having reconstructed a vector $\y^{(l)}$ with two-block support, it follows from Theorem \ref{thm:supp} that all longer periodizations $\y^{(j)}$,  $j\in\{l+1, \ldots , J-1\}$, also possess a two-block support with the same block length, so for $j>l$ we always have to apply step \ref{item:alg_1b}.

If a lower bound $2^{b-1}\leq m$ on the block length of $\y$ is known, we can begin the algorithm with the computation of $\y^{(b)}$ by applying a $2^{b}$-length IFFT algorithm to
\[
 \widehat{\y^{(b)}}=\left(\widehat{y}_{2^{J-b}k}\right)_{k=0}^{2^{b}-1},
\]
and detect its support. Then we execute the above iteration steps for $j\in\{b, \ldots , J-1\}$. The complete procedure is summarized in Algorithm \ref{alg:dft}. 

\begin{algorithm}
\caption{Sparse FFT for Vectors with Reflected Block Support}
\label{alg:dft}
\begin{algorithmic}[1]
\renewcommand{\algorithmicrequire}{\textbf{Input:}}
\renewcommand{\algorithmicensure}{\textbf{Output:}}
\Require $\widehat{\y}$, where the sought-after vector $\y\in\RR^{2N}$, $N=2^{J-1}$, has a reflected block support of (unknown) length $m$ and satisfies (\ref{eq:suppose}), and noise threshold $\eps$. If a lower bound on $m$ is known a priori, let $b\in\NN_0$ such that $2^{b-1}\leq m$, otherwise $b=0$.

\State $\y^{(b)}\gets\mathbf{IFFT}\left[\left(\widehat{y}_{2^{J-b}k}\right)_{k=0}^{2^b-1}\right]$ and, if $b>0$, detect its support structure. \label{line:init}
\For{$j$ from $b$ to $J-1$}
\Statex \begin{center}\textsc{Recovery Step for Periodizations with One-block Support} \end{center}

\If{$\y\jj$ has a one-block support}
\If{$m\jj>2^{j-1}$}
\State $\aaa\gets\diag\left(\left(\omega_{2^{j+1}}^{-l}\right)_{l=0}^{2^j-1}\right)\mathbf{IFFT}\left[\left(\widehat{y}_{2^{J-j-1}(2k+1)}\right)_{k=0}^{2^j-1}\right]$ \label{line:oneblock_direct}
\State $\left(\y\jjj_{(0)}\right)_k\gets\begin{cases}
                                         \frac{1}{2}\re\left(\y\jj+\aaa\right)_k, & \frac{1}{2}\re\left(\y\jj+\aaa\right)_k>\eps, \\
                                         0, & \text{else,} 
                                        \end{cases}$\quad $k=0,\dotsc,2^j-1$ \label{line:oneblock_direct_val1}
\State $\y\jjj_{(1)}\gets\J_{2^j}\y\jjj_{(0)}$\label{line:oneblock_direct_val2}
\ElsIf{$m\jj\leq2^{j-1}$}
\State Set $L\jj=\left\lceil\log_2m\jj\right\rceil$.
\State Set $\z\jj=\left(y\jj_{\left(\mu\jj+r\right)\mods2^j}\right)_{r=0}^{2^{L\jj}-1}$ and $\vv=\left(\widehat{y}_{2^{J-L\jj}p+2^{J-j-1}}\right)_{p=0}^{2^{L\jj}-1}$.\label{line:oneblock_short0}
\State $\aaa\gets\diag\left(\left(\omega_{2^{j+1}}^{-\left(\mu\jj+r\right)\mods2^j}\right)_{r=0}^{2^{L\jj}-1}\right)\mathbf{IFFT}\left[\diag\left(\omega_{2^{L\jj}}^{-p\mu\jj}\right)_{p=0}^{2^{L\jj}-1}\vv\right]$ \label{line:oneblock_short}
\State $\left(\z\jjj_{(0)}\right)_k\gets\begin{cases} 
                    \frac{1}{2}\re\left(\z\jj+\aaa\right)_k, & \frac{1}{2}\re\left(\z\jj+\aaa\right)_k>\eps, \\
                    0, & \text{else,}
                   \end{cases}\,k=0,\dotsc,2^{L\jj}-1$\label{line:oneblock_short_rest1}
\State $\left(y\jjj_{(0)}\right)_{\left(\mu\jj+k\right)\mods 2^j}\gets\begin{cases}
                                                                    \left(z\jjj\right)_k, & k\in\left\{0,\dotsc,2^{L\jj}-1\right\} \\
                                                                    0, & \text{else}
                                                                      \end{cases}$
\State $\left(y\jjj_{(1)}\right)_{2^j-1-\left(\mu\jj+k\right)\mods 2^j}\gets\begin{cases}
                                                                    \left(z\jjj\right)_k, & k\in\left\{0,\dotsc,2^{L\jj}-1\right\} \\
                                                                    0, & \text{else}
                                                                      \end{cases}$ \label{line:oneblock_short_rest2}
\EndIf
\Statex \begin{center}\textsc{Recovery Step for Periodizations with Two-block Support} \end{center}
\ElsIf{$\y\jj$ has a two-block support with block length $n\jj$}
\State Find $\widehat{y}_{2^{J-j-1}(2k_0+1)}\neq0$. \label{line:non-zero}
\State $\widehat{u^0}_{2 k_0+1}\gets\sum\limits_{l=\mu\jj}^{\mu\jj+n\jj-1}\omega_{2^{j+1}}^{(2 k_0+1)l}y\jj_l+\sum\limits_{l=2^{j+1}-n\jj-\mu\jj}^{2^{j+1}-1-\mu\jj}\omega_{2^{j+1}}^{(2 k_0+1)l}y\jj_{l-2^j}$ \label{line:u}
\State $\lambda\jjj\gets\begin{cases}
                         \mu\jj, & \left|\widehat{u^0}_{2 k_0+1}-\widehat{y}_{2^{J-j-1}(2 k_0+1)}\right|<\left|\widehat{u^0}_{2 k_0+1}+\widehat{y}_{2^{J-j-1}(2 k_0+1)}\right| \\
                         2^j+\mu\jj, & \text{else}
                        \end{cases}$ \label{line:lambda}
\State $y\jjj_k\gets\begin{cases}
                     y\jj_{k\mods2^j}, & k\in\left\{\lambda\jjj,\dotsc,\lambda\jjj+n\jj-1\right\}, \\
                     y\jj_{k\mods2^j}, & k\in\left\{2^{j+1}-n\jj-\lambda\jjj,\dotsc,2^{j+1}-1-\lambda\jjj\right\}, \\
                     0, & \text{else}
                    \end{cases}$ \label{line:twoblocks}
\EndIf
\State Detect the support structure of $\y\jjj$, i.e., find $\mu\jjj$ and $m\jjj$ or $n\jj$. \label{line:detection}
\EndFor
\Ensure $\y$
\end{algorithmic}
\end{algorithm}

We show that the runtime and sampling complexity of Algorithm \ref{alg:dft} are sublinear in the vector length $2N$.
\begin{thm}\label{thm:runtime}
Let $N=2^{J-1}$ and $\x\in\RR^N$ have a one-block support of length $m<N$. Set $\y=(\x^T,(\J_N\x)^T)^T$ and assume that $\y$ satisfies $(\ref{eq:suppose})$. Further, assume that there is no a priori knowledge of the support length $m$ of $\x$. Then Algorithm $\ref{alg:dft}$ has a runtime of $\mathcal{O}\left(m\log m\log \frac{2N}{m}\right)$ and uses $\mathcal{O}\left(m\log \frac{2N}{m}\right)$ samples of $\widehat{\y}$.
\end{thm}
\begin{proof}
1. Note that the support $S^{(J)}$ of $\y=\y^{(J)}$ has at most cardinality $2m$. Let $2^{L-1} < 2m \le 2^{L} \le 2^{J}$. For $j\in\{0, \ldots , L-1\}$ the vector $\y\jj$ has necessarily a one-block support of length $2^{j-1} \le m^{(j)}\leq 2^{j}$, and we have to apply the procedure of Theorem \ref{thm:oneblock}, where the computation of $\y^{(j+1)}_{(0)} = {\mathbf z}^{(j+1)}_{(0)}$ requires an IFFT of length $2^{j}$ with $\mathcal{O}\left(2^j\log2^j\right)$ operations according to line \ref{line:oneblock_direct} of Algorithm \ref{alg:dft}.
To determine the nonzero entries of $\y^{(j+1)}$ in lines \ref{line:oneblock_direct_val1} and \ref{line:oneblock_direct_val2}, and to detect its support structure in line \ref{line:detection}, only $\mathcal{O}\left(2^j \right)$ operations are needed, as we will show in Section \ref{sec:detection}.
Thus, the iteration steps for $j\in\{0, \ldots , L-1\}$ have a runtime complexity of 
\[
 {\mathcal O} \left(\sum_{j=0}^{L-1} 2^{j} \log 2^{j} \right) = {\mathcal O}\left(2^{L}(L-2)\right) .
\]
For $j\in\{L, \ldots , J-1\}$ we have to apply either the recovery step for periodizations with one-block support or the recovery step for periodizations with two-block support.
If $\y^{(j)}$ has a one-block support of length $m^{(j)}$, then $m \le m\jj \le 2m$ and $2^{L-1}< m^{(j)} \le 2^{L}$. 
The computation of ${\mathbf z}^{(j+1)}_{(0)}$ in lines \ref{line:oneblock_short} and \ref{line:oneblock_short_rest1} requires an IFFT of length at most $2^{L}$ and further operations of complexity ${\mathcal O}\left(2^{L}\right)$. In order to detect the support structure in line \ref{line:detection} at most ${\mathcal O}\left(2^{L}\right)$ operations are necessary. Altogether, we require ${\mathcal O}\left(2^{L} \log 2^{L}\right)$ operations for such an iteration step.

If $\y^{(j)}$ has a two-block support with block length $n\jj=m$, executing lines \ref{line:non-zero} to \ref{line:twoblocks} requires ${\mathcal O}(m)$ operations, and the support structure of $\y^{(j+1)}$ is already completely determined. However, in the worst case, we have to apply the recovery step for periodizations with one-block support for every $j\in\{L, \ldots , J-1\}$, and thus need ${\mathcal O}\left((J-L) 2^{L} \log 2^{L}\right)$ operations for these steps. Adding the arithmetical complexities for both cases yields an overall runtime of  
\[
 {\mathcal O}\left(( 2^{L}(L-2)  + (J-L) 2^{L} L\right) = {\mathcal O} \left( m \log m \log\frac{2N}{m} \right),
\]
where we have used that $2m \le 2^{L} <4m$. In particular, if $m$ approaches $N$, the algorithm has a runtime of $\mathcal{O}(N\log N)$, which is the same as the runtime of a full length FFT.

2. For $j\in\{0,\dotsc,L-1\}$ it follows from $2^{L-1}< 2m \le 2^{L}$ that the computation of $\y\jjj$ requires all $2^{j-1}$ samples of $\left(\widehat{y\jjj}_{2k+1}\right)_{k=0}^{2^{j-1}-1}$. This implies that we need the entire vector $\widehat{\y^{(L)}}$ to recover $\y^{(L)}$ iteratively. The remaining iteration steps with $j\in\{L,\dotsc,J-1\}$ need $2^{L\jj}=\mathcal{O}\left(m\jj\right)$ samples if $\y\jj$ has a one-block support. In the case of a two-block support it suffices to examine $2m$ samples of $\widehat{\y}$ to find a nonzero one by Lemma \ref{lem:nonzero}. Hence, the algorithm has a sampling complexity of 
\[
 \mathcal{O}\left(\sum_{j=0}^{L-1}2^j+(J-L)2^L\right)=\mathcal{O}\left(2^L+(J-L)2^L\right)=\mathcal{O}\left(m\log\frac{2N}{m}\right).
\]
\end{proof}
\subsection{Detecting the Support Sets}\label{sec:detection}
Algorithm \ref{alg:dft} relies heavily on an efficient detection of the support structure of the periodized vector $\y^{(j)}$.
Let us assume that $\y\jj$ and its support are known, and that we computed $\y\jjj$ with the appropriate method in Algorithm \ref{alg:dft}. In order to detect the support of $\y\jjj$ in line \ref{line:detection} of Algorithm \ref{alg:dft} we have to distinguish the five cases already considered in Theorem \ref{thm:supp}. Note that for noisy data the found block lengths $n\jjj$ for $\y\jjj$ with two-block support might not be the same as the exact block length $m$ of $\x$.
\begin{enumerate}[wide, labelwidth=!, labelindent=0pt]
 \labitemc{$\text{A}_\text{1}$}{item:detect_middle} 
 $\y\jj$ has the one-block support $S\jj=I_{\mu^{(j)},2^{j}-1-\mu^{(j)}}$ of length $m^{(j)}<2^j$ centered around the middle of the vector, i.e., around $2^{j-1}-1$ and $2^{j-1}$.
 
 Then $\y\jjj$ has a two-block support of length $n\jjj\leq m\jj$ and the support of the first block is a subset of 
 $S\jj$ by Theorem \ref{thm:supp}, case \ref{item:thm1_middle}. We define  
 \[
  T\jjj_{(0)}\coloneqq\left\{k\in\left\{\mu\jj,\mu\jj+1,\dotsc,2^j-1-\mu\jj\right\}:y\jjj_k>\eps\right\}
  \eqqcolon\left\{t_1,\dotsc,t_{K}\right\},
 \]
 where $t_1<\dotsb<t_{K}$, and $\eps>0$ depends on the noise level. We have to check $m\jj$ entries of $\y\jjj$ in order to determine $S\jjj$. 
 Then we choose
 \[
  \mu\jjj \coloneqq t_1 \quad\text{and}\quad n\jjj\coloneqq t_{K}-t_1+1.
 \] 
 \labitemc{$\text{A}_\text{2}$}{item:detect_full}
 If $j<J-1$, $\mu^{(j)}=0$ and $m\jj=2^j$, Theorem \ref{thm:supp}, case \ref{item:thm1_full} yields that $\y\jjj$ has a one-block support whose location and length are unknown. As $\y\jjj$ is symmetric, the indices of its significantly large entries are
  \[
   T\jjj\coloneqq\left\{t_1,\dotsc,t_K,2^{j+1}-1-t_K,\dotsc,2^{j+1}-1-t_1\right\}\eqqcolon\left\{t_1,\dotsc,t_{2K}\right\},
  \]
  where $t_K\leq2^j-1$ and $t_{K+1}\geq2^j$. If $t_{2K}-t_1+1>2^j$, we set $\mu\jjj\coloneqq 0$ and $m\jjj\coloneqq 2^{j+1}$. Otherwise we need to detect whether the support block is centered around the middle or the boundary of $\y\jjj$. We define
  \[
   d_0\coloneqq t_{K+1}-t_K \quad\text{and}\quad d_1\coloneqq\left(t_1-t_{2K}\right)\mods 2^{j+1}.
  \]
  If $d_0<d_1$, then $\y\jjj$ has a one-block support centered around $2^j-1$ and $2^j$. If $d_0>d_1$, the support block is centered around 0 and $2^{j+1}-1$. If $d_0=d_1$, $\y\jjj$ must have full support of length $m\jjj=2^{j+1}$, so we can conclude
  \[
   \mu\jjj\coloneqq\begin{cases}
            t_1, & d_0<d_1, \\
            t_{K+1}, & d_0>d_1, \\
            0, & d_0=d_1,
           \end{cases} \quad\text{and}\quad  
           m\jjj\coloneqq\begin{cases}
            t_{2K}-t_1+1, & d_0<d_1, \\
            2^{j+1}-t_{K+1}+t_K+1, & d_0>d_1, \\
            2^{j+1}, & d_0=d_1.
           \end{cases}
  \]
 In the case that $j=J-1$ and $m^{(J-1)}=2^{J-1}$, let
 \begin{align*}
  T^{(J)}_{(0)}&\coloneqq\left\{k\in\left\{0,\dotsc,2^{J-2}-1\right\}:y_k>\eps\right\}\eqqcolon\left\{t_1,\dotsc,t_K\right\} \qquad \text{and}\quad \\
  T^{(J)}_{(1)}&\coloneqq\left\{k\in\left\{2^{J-2},\dotsc,2^{J-1}-1\right\}:y_k>\eps\right\}\eqqcolon\left\{u_1,\dotsc,u_L\right\}.
 \end{align*}
 If $T^{(J)}_{(0)}=\emptyset$, then $\y$ has a one-block support centered around the middle with
 \[
  \mu^{(J)}\coloneqq u_1 \quad\text{and}\quad m^{(J)}\coloneqq2\left(2^{J-1}-u_1\right)=2m.
 \]
 If $T^{(J)}_{(1)}=\emptyset$, then $\y$ has a one-block support centered around the boundary with
 \[
  \mu^{(J)}\coloneqq 2^J-1-t_K \quad\text{and}\quad m^{(J)}\coloneqq2\left(t_K+1\right)=2m.
 \]
 Otherwise there are three possibilities for the support of $\y$, where we do not know the correct one a priori:
 \begin{enumerate}
  \item $S^{(J)}\coloneqq I_{t_1,2^J-1-t_1}$,
  \item $S^{(J)}\coloneqq I_{2^J-1-u_L,u_L}$,
  \item $\y$ has a two-block support with two blocks of possibly different lengths and unknown positions.
 \end{enumerate}
 
 \labitemc{$\text{A}_{\text{3}}$}{item:detect_bound} $j< J-1$ and $\y\jj$ has the one-block support $S\jj=I_{\mu^{(j)},2^{j}-1-\mu^{(j)}}$  of length $m\jj<2^j$ centered around the boundary of the vector, i.e., around $2^j-1$ and $0$.
  
  It follows from Theorem \ref{thm:supp}, case \ref{item:thm1_bound} that $\y\jjj$ has a one-block support of length $m\jjj \coloneqq m\jj$ with $\mu\jjj=\mu\jj$ or $\mu\jjj=2^j+\mu\jj$. We compare the entries at the possible locations of the support block and set
 \[
  e_0\coloneqq\sum_{k=\mu\jj}^{\mu\jj+m\jj-1}\left|y\jjj_k\right| \quad\text{and}\quad 
  e_1\coloneqq\sum_{k=2^j+\mu\jj}^{\left(2^j+\mu\jj+m\jj-1\right)\mods2^{j+1}}\left|y\jjj_k\right|.
 \] 
 Since for exact data one of the sums has only vanishing summands, we choose
 \[
  \mu\jjj \coloneqq\begin{cases}
           \mu\jj, & e_0>e_1 \\
           2^j+\mu\jj, & e_0<e_1.
          \end{cases}
 \] 
 \labitemc{$\text{A}_{\text {4}}$}{item:detect_bound_J} $j=J-1$ and  $\y^{(J-1)}$ has the one-block support $S^{(J-1)}=I_{\mu^{(J-1)},2^{J-1}-1-\mu^{(J-1)}}$ of length $m^{(J-1)}<2^{J-1}$ centered around the boundary, i.e., around $2^{J-1}-1$ and $0$.
  
 Theorem \ref{thm:supp}, case \ref{item:thm1_bound_J} implies that $\y^{(J)}=\y$ has either a two-block support with two separated blocks of possibly different lengths or, as a boundary case, a one-block support, where one of the two blocks in (\ref{eq:thm1_bound_J}) vanishes. In case of two blocks, one is centered  around the middle of the vector, i.e., around $2^{J-1}-1$ and $2^{J-1}$, and its support is a subset of $I_{\mu^{(J-1)}, 2^{J}-1-\mu^{(J-1)}}$, and the other is centered around the boundary of the vector, i.e., around $2^J-1$ and $0$, and its support is a subset of $I_{2^{J-1}+\mu^{(J-1)}, 2^{J-1}-1- \mu^{(J-1)}}$. If $\y$ has a one-block support, its support is a subset of one of the two index sets above. Since the support blocks have even length, set
 \begin{align*}
  T_{(0)}^{(J)}&\coloneqq\left\{k\in\left\{\mu^{(J-1)},\dotsc,2^{J-1}-1\right\}:y_k>\eps\right\}\eqqcolon\left\{t_1,\dotsc,t_K\right\} \quad \text{and} \\
  T_{(1)}^{(J)}&\coloneqq\left\{k\in\left\{2^{J-1}+\mu^{(J-1)},\dotsc,2^{J}-1\right\}:y_k>\eps\right\}\eqqcolon\left\{u_1,\dotsc,u_L\right\}.
 \end{align*}
If $T_{(0)}^{(J)}=\emptyset$, then $\y$ has a one-block support centered around the boundary, and we set
 \[
  \mu^{(J)}\coloneqq u_1 \quad\text{and}\quad m^{(J)}\coloneqq2\cdot\left(2^J-u_1\right)=2m
 \]
 to obtain the support set $S^{(J)}\coloneqq I_{\mu^{(J)}, 2^J-1-\mu^{(J)}}$ of $\y$.
 If $T_{(1)}^{(J)}=\emptyset$, then $\y$ has a one-block support centered around the middle of the vector, and we let
 \[
  \mu^{(J)}\coloneqq t_1 \quad\text{and}\quad m^{(J)}\coloneqq 2\cdot\left(2^{J-1}-t_1\right)=2m,
 \] 
 implying that the support set of $\y$ is $S^{(J)}\coloneqq I_{\mu^{(J)}, \mu^{(J)}+ m^{(J)} - 1}$.
 If neither set is empty, $\y$ has a two-block support with two separated blocks of possibly different lengths, see Figure \ref{fig2} (left). We denote the first index of the block centered around the middle by $\mu^{(J)}$ and the first index of the block centered around the boundary by $\eta^{(J)}$ and set
 \[
  \mu^{(J)}\coloneqq t_1 \quad\text{and}\quad \eta^{(J)} \coloneqq u_{1},
 \]
 We obtain the support set $S^{(J)}\coloneqq I_{\mu^{(J)},2^{J}-1-\mu^{(J)}} \cup I_{\eta^{(J)},2^J-1-\eta^{(J)}}$, where the block centered around the middle of the vector has length $n^{(J)}_{(0)}\coloneqq2\left(2^{J-1}-t_1\right)$ and the block centered around the boundary has length $n^{(J)}_{(1)}\coloneqq2\left(2^J-u_1\right)$.
 
 \labitemc{B}{item:detect_b}  $\y\jj$ has the two-block support $S\jj=I_{\mu^{(j)},\nu^{(j)}} \cup I_{2^{j}-1-\nu^{(j)},2^{j}-1-\mu^{(j)}}$ with block length $n\jj=m$.

Let  $\lambda\jjj$ be the support index  computed in line \ref{line:lambda} of Algorithm \ref{alg:dft}. We obtain 
 \[
  \mu\jjj\coloneqq\begin{cases}
           \mu\jj, & \lambda\jjj=\mu\jj, \\
           2^j-m-\mu\jj, & \lambda\jjj=2^j+\mu\jj,
          \end{cases}
 \]
 according to Theorem \ref{thm:supp}, case \ref{item:thm1_b} and Figure \ref{fig2} (right).
\end{enumerate}

\subsection{A Sparse Fast DCT for Vectors with One-block Support}\label{sec:dct_alg}
We now apply the sparse FFT algorithm for vectors $\y \in {\mathbb R}^{2N}$ with reflected block support presented in Section \ref{sec:dft_alg} to derive a sparse fast DCT algorithm for vectors with one-block support.
Recall that by Lemma \ref{lem:dctx_from_yhat} the Fourier transform of the vector $\y=(\x^T,(\J_N\x)^T)^T$ is completely determined by $\cxx$. Hence, we can compute $\y$ from $\widehat{\y}$ with the help of Algorithm \ref{alg:dft} if $\cxx$ is known. By construction, $\x$ is then given as the first half of $\y$. Since Algorithm \ref{alg:dft} is adaptive, no a priori knowledge of the support length of $\x$ is required. The resulting sparse fast DCT procedure is summarized in Algorithm \ref{alg:dct}. 
\begin{algorithm}
\caption{Sparse Fast DCT for Vectors with One-block Support}
\label{alg:dct}
\begin{algorithmic}[1]
\renewcommand{\algorithmicrequire}{\textbf{Input:}}
\renewcommand{\algorithmicensure}{\textbf{Output:}}
\Require $\cxx$, where the sought-after vector $\x\in\RR^{N}$, $N=2^{J-1}$, has a one-block support of (unknown) length $m$ and $\y=(\x^T,(\J_N\x)^T)^T$ satisfies (\ref{eq:suppose}), and noise threshold $\eps$. If a lower bound on $m$ is known a priori, let $b\in\NN_0$ such that $2^{b-1}\leq m$, otherwise $b=0$.

\State Compute $\widehat{y}_k=\begin{cases}
                           \frac{\sqrt{2N}}{\eps_N(k)}\omega_{4N}^{-k}\cdot\cx_k, & k\in\{0,\dotsc,N-1\}, \\
                           0, & k=N, \\
                           -\frac{\sqrt{2N}}{\eps_N(2N-k)}\omega_{4N}^{-k}\cdot\cx_{2N-k}, & k\in\{N+1,\dotsc,2N-1\}
                          \end{cases}$ \label{line:init_dct} \newline
        if the sample $\widehat{y}_{k} $ is needed in Algorithm  \ref{alg:dft}.              
\State $\y\gets\text{ Algorithm \ref{alg:dft}}\left[\widehat{\y},b\right]$ \label{line:apply}
\State $\x\gets\y_{(0)}=\left(y_k\right)_{k=0}^{N-1}$ \label{line:output}
\Ensure $\x$
\end{algorithmic}
\end{algorithm}
\begin{thm}
 Let $N=2^{J-1}$ and $\x\in\RR^N$ have a one-block support of length $m<N$. Set $\y=(\x^T,(\J_N\x)^T)^T$ and assume that $\y$ satisfies $(\ref{eq:suppose})$. Further, assume that there is no a priori knowledge of the support length $m$ of $\x$. Then Algorithm $\ref{alg:dct}$ has a runtime of $\mathcal{O}\left(m\log m\log \frac{2N}{m}\right)$ and uses $\mathcal{O}\left(m\log \frac{2N}{m}\right)$ samples of $\cxx$.
\end{thm}
\begin{proof}
As shown in Theorem \ref{thm:runtime}, Algorithm \ref{alg:dft} requires $\mathcal{O}\left(m\log\frac{2N}{m}\right)$ samples of $\widehat{\y}$ in lines \ref{line:init}, \ref{line:oneblock_direct}, \ref{line:oneblock_short0} and \ref{line:non-zero}, so we also need $\mathcal{O}\left(m\log\frac{2N}{m}\right)$ samples of $\cxx$ in order to compute the necessary samples of $\widehat{\y}$ in line \ref{line:init_dct} of Algorithm \ref{alg:dct} by Lemma \ref{lem:dctx_from_yhat}. Line \ref{line:init_dct} requires $\mathcal{O}\left(m\log\frac{2N}{m}\right)$ operations; hence the runtime of Algorithm \ref{alg:dct} is governed by the runtime of Algorithm \ref{alg:dft}.
\end{proof}

\section{Numerical results}\label{sec:numerics}
\subsection{Numerical Results for the Reflected Block Support FFT}\label{sec:numerics_dft}
In the following section we present some test results regarding the runtime of Algorithm \ref{alg:dft} and its robustness to noise. We compare our method to Algorithm 2.3 in \cite{plonka_sparse} and to \textsc{Matlab}'s \texttt{ifft} routine, which is a fast and highly optimized implementation of the fast inverse Fourier transform, based on the FFTW library, see \cite{fftw,matlab_ifft,matlab_fft}. The former two algorithms have been implemented in \textsc{Matlab}, and the code is freely available in \cite{nam_code_refl_dft,nam_code_msparse}. 
The sparse FFT algorithm in \cite{plonka_sparse} is suitable for the fast reconstruction of an $m$-sparse vector $\y$ from its DFT $\widehat{\y}$ for small $m$. Note that neither of the algorithms requires a priori knowledge of the length of the support blocks or the sparsity, but Algorithm \ref{alg:dft} needs that $\y$ satisfies (\ref{eq:suppose}).

Figure \ref{fig:runtime} shows the average runtimes of Algorithm \ref{alg:dft} with threshold $\eps=10^{-4}$, Algorithm 2.3 in \cite{plonka_sparse}, in the variant using Algorithm 4.5, and \texttt{ifft} applied to $\widehat{\y}$ for 100 randomly generated  vectors $\y$ of length $2N=2^{21}$ with reflected block support of lengths varying between 5 and 50000. The nonzero entries of the vectors are chosen between 0 and 10. For each vector at most $\lfloor (m-2)/2 \rfloor$ entries in the first support block, excluding the first and last one, are randomly set to 0, and the second half of $\y$ is determined by its symmetry $\y = \J_{2N} \y$.
\begin{figure}[!ht]
\centering
\input{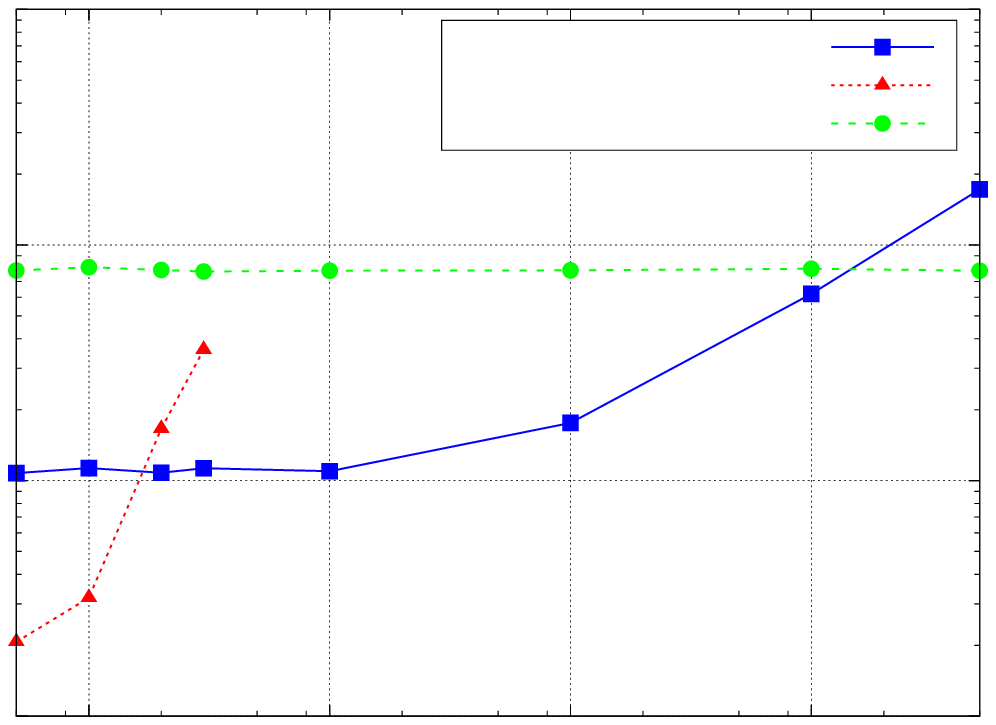}
\caption{Average runtimes of Algorithm \ref{alg:dft} with $\eps=10^{-4}$, Algorithm 2.3 (using Algorithm 4.5) in \cite{plonka_sparse} and \textsc{Matlab}'s \texttt{ifft} for 100 random input vectors with reflected block support of length $m$ and vector length $2N=2^{21}$.} 
\label{fig:runtime}
\end{figure}
Algorithm 2.3 in \cite{plonka_sparse} is very unstable for greater sparsities $2m$, as it often has to solve a close to singular equation system. Hence, we decided to only measure its runtime for block lengths up to $m=30$. Obviously, any comparison to the highly optimized \texttt{ifft} routine must be flawed; however, we can see that both Algorithm \ref{alg:dft} and Algorithm 2.3 in \cite{plonka_sparse} are much faster than \texttt{ifft} for sufficiently small block lengths. The former algorithm achieves faster runtimes for block lengths up to $m=10000$, while the latter does so at least for block lengths up to $m=30$. 
Note that by setting at most $\lfloor(m-2)/2\rfloor\cdot 2$ entries inside the support blocks randomly to 0, the actual sparsity of $\y$ can be almost as low as $m$. This barely affects the runtime of Algorithm \ref{alg:dft}, but it decreases the average runtime of Algorithm 2.3 in \cite{plonka_sparse}. As can be seen from Table \ref{tab:error_ex} presenting the average reconstruction errors for exact data, Algorithm 2.3 in \cite{plonka_sparse} is not accurate for block lengths of $m=20$ or greater, and, as we found out during the experiments, not even consistently accurate for block lengths up to $m=10$. 
\begin{table}[!ht]
\begin{center}
\begin{tabular}{rccc}\toprule
$m$	& Algorithm \ref{alg:dft}	& Algorithm 2.3 in \cite{plonka_sparse}	& \texttt{ifft} \\ \midrule
5	& $4.2\cdot 10^{-20}$	& $\;\ 3.4\cdot 10^{-8}$	& $3.8\cdot 10^{-21}$ \\
10 	& $8.0\cdot 10^{-20}$	& $1.4\cdot 10^{0}$	& $4.8\cdot 10^{-21}$ \\
20	& $2.2\cdot 10^{-19}$	& $3.1\cdot 10^{7}$	& $7.0\cdot 10^{-21}$ \\
30	& $6.6\cdot 10^{-19}$	& $3.9\cdot 10^{8}$	& $8.3\cdot 10^{-21}$ \\
100	& $1.5\cdot 10^{-18}$	& $-$	& $1.5\cdot 10^{-20}$ \\
1000	& $7.7\cdot 10^{-14}$	& $-$	& $4.7\cdot 10^{-20}$ \\
10000	& $3.6\cdot 10^{-12}$	& $-$	& $1.5\cdot 10^{-19}$ \\
50000	& $1.3\cdot 10^{-11}$	& $-$	& $3.5\cdot 10^{-19}$ \\  \bottomrule 
\end{tabular}
\caption{Reconstruction errors for the three DFT algorithms for exact data.}
\label{tab:error_ex}
\end{center}
\end{table}
Still, for block lengths up to $m=10$, this procedure is faster than Algorithm \ref{alg:dft}. For block lengths up to $m=100$ Algorithm \ref{alg:dft} achieves an accuracy close to that of \texttt{ifft}, and even for $m=50000$ its reconstruction error is small.

Next we examine the robustness of the algorithms for noisy data. Since Algorithm 2.3 in \cite{plonka_sparse} is not suitable for noisy data due to ill-conditioned equation systems having to be solved, we will only consider Algorithm \ref{alg:dft} and \textsc{Matlab}'s \texttt{ifft} hereafter. Disturbed Fourier data $\widehat{\z}\in\CC^{2N}$ is created by adding uniform noise $\bfeta\in\CC^{2N}$ to the given data $\widehat{\y}$,
\[
 \widehat{\z}\coloneqq\widehat{\y}+\bfeta.
\]
We measure the noise with the SNR value,
\[ 
 \text{SNR}\coloneqq20\cdot\log_{10}\frac{\left\|\widehat{\y}\right\|_2}{\left\|\bfeta\right\|_2}.
\]
Figures \ref{fig:error_100} and \ref{fig:error_1000} depict the average reconstruction errors $\|\y-\y'\|_2/(2N)$ for block lengths $m=100$ and $m=1000$, where $\y$ denotes the original vector and $\y'$ the reconstruction by the corresponding algorithm applied to $\widehat{\z}$.  Note that for noisy data the resulting vector  $\y'$ does no longer have an exact reflected block support, but the support blocks have entries that are significantly greater than the noise and can thus be found  by the support detection procedures presented in Section \ref{sec:detection}.
\begin{figure}[!tbp]
  \centering
  \subfloat[Reconstruction error for $m=100$.]{\resizebox{0.49\textwidth}{!}{\input{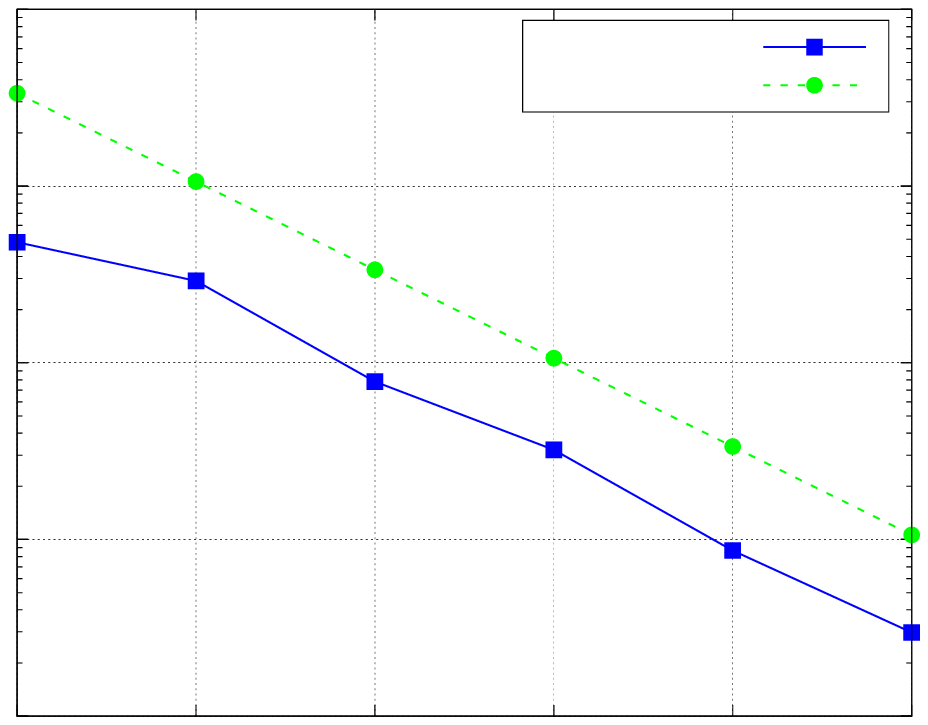}}\label{fig:error_100}}
  \hfill
  \subfloat[Reconstruction error for $m=1000$.]{\resizebox{0.49\textwidth}{!}{\input{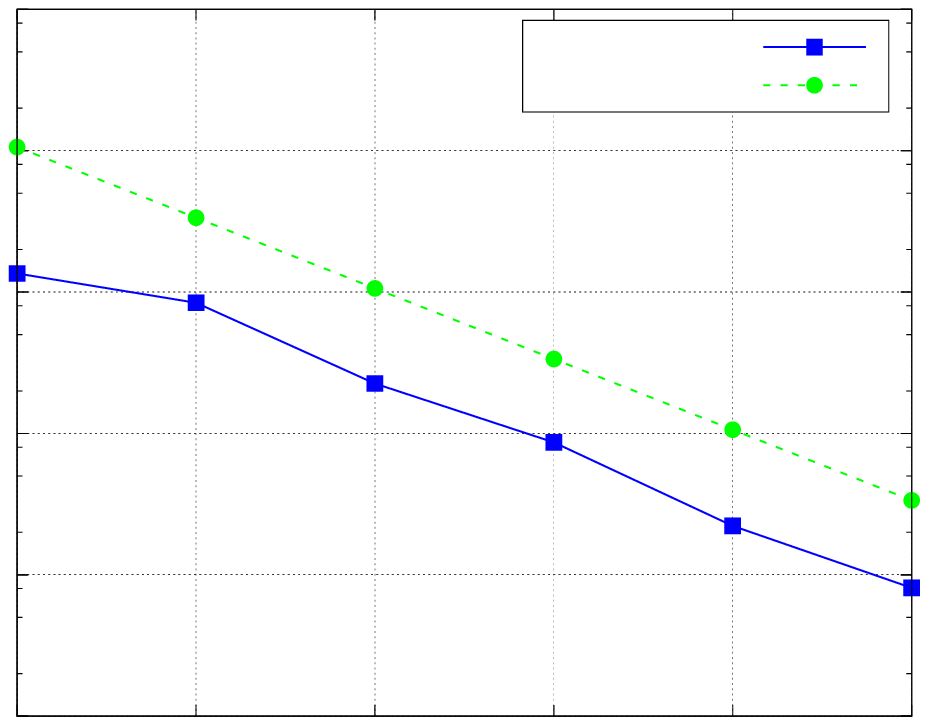}}\label{fig:error_1000}}
  \caption{Average reconstruction errors $\|\y-\y'\|_2/(2N)$ of Algorithm \ref{alg:dft} and \texttt{ifft} for 100 random input vectors with support length $m$ and vector length $2N=2^{21}$.}
\end{figure}
The threshold $\varepsilon$ for Algorithm \ref{alg:dft} is chosen according to Table \ref{tab:epsilon}.
\begin{table}[!ht]
\begin{center}
\begin{tabular}{ccccccc}\toprule
 SNR	& 0 & 10 & 20 & 30 & 40 & 50  \\ \midrule
 $\varepsilon$	& 1.7 & 1.2 & 0.4 & 0.19 & 0.05 & 0.02 \\  \bottomrule 
\end{tabular}
\caption{Threshold $\varepsilon$ for Algorithm \ref{alg:dft}.}
\label{tab:epsilon}
\end{center}
\end{table}
These values were found via an attempt to minimize the approximation error and maximize the rate of correct recovery. Both for $m=100$ and $m=1000$ we see that the reconstruction by Algorithm \ref{alg:dft} yields a smaller error than the one by \texttt{ifft} for all considered noise levels. 

Since for vectors with reflected block support the structure is especially important, we also examine whether Algorithm \ref{alg:dft} can correctly identify the support blocks of $\y$ for noisy input data. 
Especially  for high noise levels, Algorithm \ref{alg:dft} tends to overestimate  the true length of the support blocks. Table \ref{tab:error_dft} shows the rates of correct recovery of the support. 
\begin{table}[!ht]
\begin{center}
\begin{tabular}{ccccc}\toprule
 & \multicolumn{4}{c}{Rate of Correct Recovery in \% Using Algorithm \ref{alg:dft} for} \\ 
 \multirow{2}{*}{SNR}	& $m=100$ & $m=100$ & $m=1000$ & $m=1000$ \\
 & & $m'\leq3m=300$ & & $m'\leq3m=3000$ \\ \midrule
 0	& 70	& 49	& 69	& 47	\\
10	& 70	& 70	& 74	& 68	\\
20	& 86	& 83	& 93	& 85 	\\
30	& 98	& 98	& 94	& 93 	\\
40	& 99	& 98	& 97	& 93	\\
50	& 100	& 100	& 99	& 98	\\ \bottomrule 
\end{tabular}
\caption{Rate of correct recovery of the support of $\y$ in percent for Algorithm \ref{alg:dft}, without bounding $m'$ and with $m'\leq3m$, for the 100 random input vectors with block length $m=100$ and $m=1000$ from Figures \ref{fig:error_100} and \ref{fig:error_1000}.}
\label{tab:error_dft}
\end{center}
\end{table}
In the second and fourth column we present the rate of correct recovery where we consider $\y$ to be correctly recovered by $\y'$ if the support of the two blocks of the original vector $\y$ is contained in the support blocks found by the algorithm. In the 
third and fifth column  we additionally require that the block length $m'$ found by Algorithm \ref{alg:dft} satisfies  $m'\leq 3m$. 

For SNR values of 20 or more the algorithm has a very high rate of correct recovery in the sense that the original support is contained in the reconstructed one. For these noise levels the block length of $\y'$ is almost always at most $3m$. 

\subsection{Numerical Results for the Fast One-block Support DCT}
In this section we present the results of numerical experiments regarding runtime and the robustness to noise of Algorithm \ref{alg:dct}. As, to the best of our knowledge, there exist no other sparse DCT algorithms, we only compare our method to \textsc{Matlab}'s \texttt{idct} routine, which is contained in the \emph{Signal Processing Toolbox}, see \cite{matlab_idct}. \texttt{idct} is a fast and highly optimized implementation of the fast inverse cosine transform of type II. Our algorithm has been implemented in \textsc{Matlab}, and the code is freely available in \cite{nam_code_dct}. Note that Algorithm \ref{alg:dct} also does not require any a priori knowledge of the support length but needs that $\y=(\x^T,(\J_N\x)^T)^T$ satisfies (\ref{eq:suppose}).

In Figure \ref{fig:runtime_dct} one can see the average runtimes of Algorithm \ref{alg:dct} with threshold $\eps=10^{-4}$ and \texttt{idct} applied to $\cxx$ for 100 randomly generated $2^{20}$-length vectors $\x$ with one-block support of lengths varying between 10 and 50000. As in Section \ref{sec:numerics_dft}, the entries of the input vectors are between 0 and 10, and at most $\lfloor(m-2)/2\rfloor$ entries are randomly set to 0.
\begin{figure}[!ht]
\centering
\input{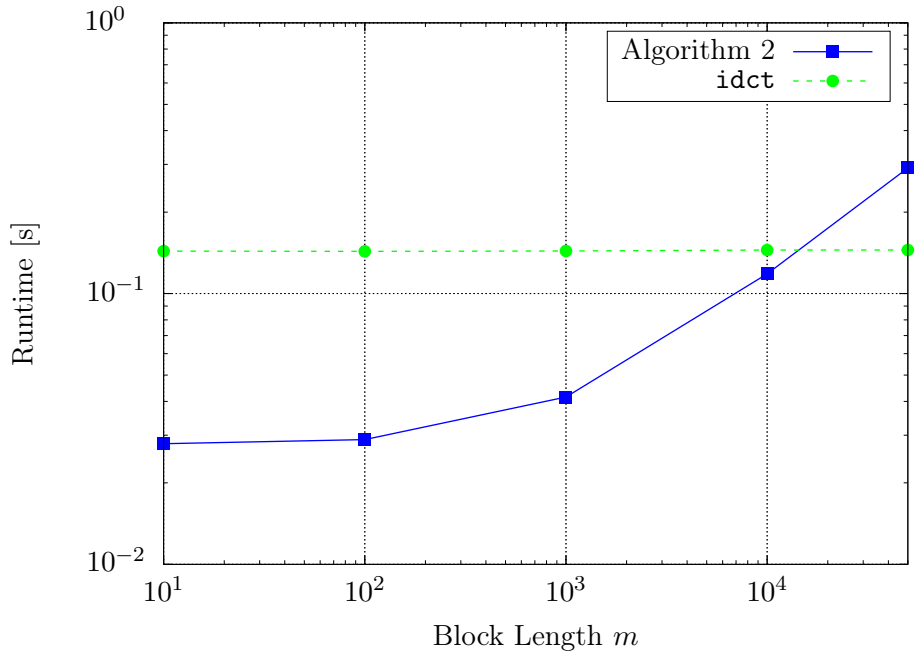}
\caption{Average runtimes of Algorithm \ref{alg:dct} with $\eps=10^{-4}$ and \textsc{Matlab}'s \texttt{idct} for 100 random input vectors with support of length $m$ and vector length $N=2^{20}$.} 
\label{fig:runtime_dct}
\end{figure}
Again, comparing Algorithm \ref{alg:dct} to the highly optimized \texttt{idct} routine is flawed, but, at least for block lengths up to $m=10000$, our method is faster than \texttt{idct} and, as shown in Table \ref{tab:error_ex_dct}, achieves reconstruction errors comparable to those of \texttt{idct} for block lengths up $m=100$ and still a very high accuracy for greater block lengths. 
\begin{table}[!ht]
\begin{center}
\begin{tabular}{rcc}\toprule
$m$	& Algorithm \ref{alg:dct}	& \texttt{idct} \\ \midrule
10	& $9.6\cdot 10^{-20}$	& $7.1\cdot 10^{-21}$ \\
100	& $4.7\cdot 10^{-18}$	& $2.2\cdot 10^{-20}$ \\
1000 	& $1.4\cdot 10^{-16}$	& $6.9\cdot 10^{-20}$ \\
10000 	& $2.5\cdot 10^{-12}$	& $2.2\cdot 10^{-19}$ \\
50000 	& $1.7\cdot 10^{-11}$	& $4.9\cdot 10^{-19}$ \\  \bottomrule 
\end{tabular}
\caption{Reconstruction errors for Algorithm \ref{alg:dct} and \texttt{idct} for exact data.}
\label{tab:error_ex_dct}
\end{center}
\end{table}

We also investigate the robustness of Algorithm \ref{alg:dct} for noisy data. As before we create disturbed cosine data $\z^{\widehat{\mathrm{II}}}\in\RR^{N}$ by adding uniform noise $\bfeta\in\RR^{N}$ to the given data $\cxx$,
\[
 \z^{\widehat{\mathrm{II}}}\coloneqq\cxx+\bfeta.
\]
Then the SNR is given by
\[ 
 \text{SNR}\coloneqq20\cdot\log_{10}\frac{\left\|\cxx\right\|_2}{\left\|\bfeta\right\|_2}.
\]
In Figures \ref{fig:error_dct_100} and \ref{fig:error_dct_1000} one can see the average reconstruction errors 
$\left\|\x-\x'\right\|_2/N$,
where $\x$ denotes the original vector and $\x'$ the reconstruction by the corresponding algorithm applied to $\z^{\widehat{\mathrm{II}}}$ for block lengths $m=100$ and $m=1000$. 
\begin{figure}[!tbp]
  \centering
  \subfloat[Reconstruction error for $m=100$.]{\resizebox{0.49\textwidth}{!}{\input{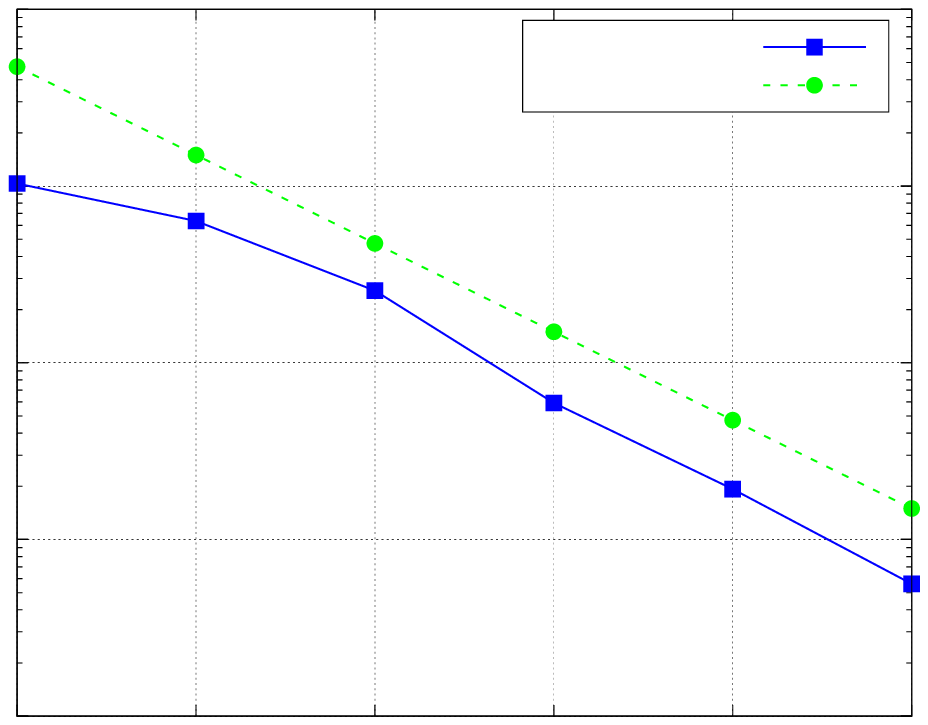}}\label{fig:error_dct_100}}
  \hfill
  \subfloat[Reconstruction error for $m=1000$.]{\resizebox{0.49\textwidth}{!}{\input{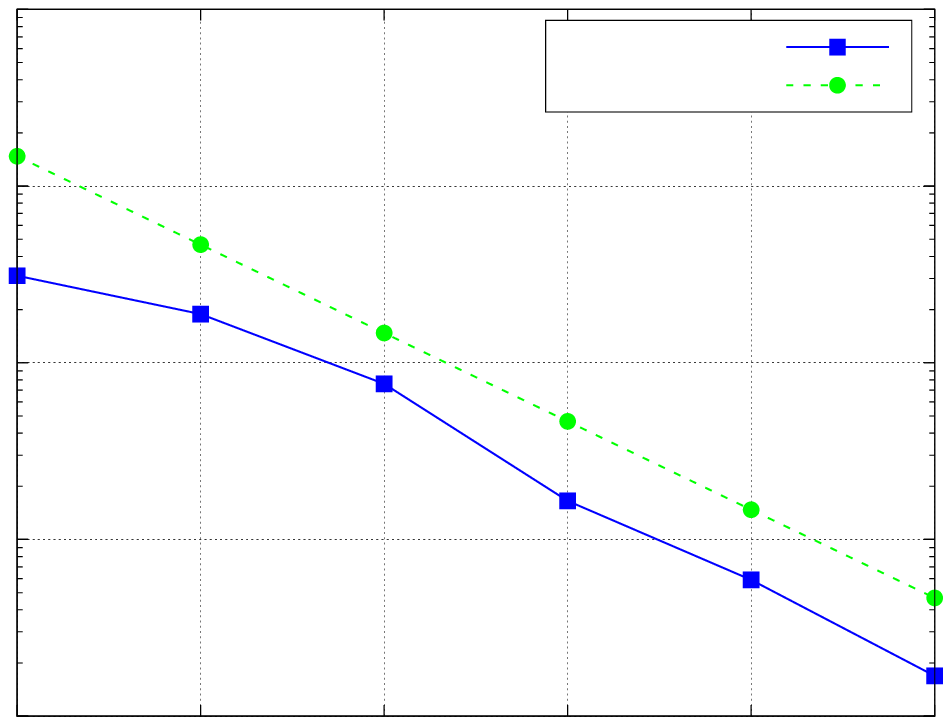}}\label{fig:error_dct_1000}}
  \caption{Average reconstruction errors $\|\x-\x'\|_2/N$ of Algorithm \ref{alg:dct} and \texttt{idct} for 100 random input vectors with support length $m$ and vector length $N=2^{20}$.}
\end{figure}
The threshold parameter $\varepsilon$ for Algorithm \ref{alg:dct} is chosen according to Table \ref{tab:epsilon_dct}.
\begin{table}[!ht]
\begin{center}
\begin{tabular}{ccccccc}\toprule
 SNR	& 0 & 10 & 20 & 30 & 40 & 50  \\ \midrule
 $\varepsilon$	& 2.5 & 1.8 & 1 & 0.3 & 0.15 & 0.05 \\  \bottomrule 
\end{tabular}
\caption{Threshold $\varepsilon$ for Algorithm \ref{alg:dct}.}
\label{tab:epsilon_dct}
\end{center}
\end{table}
Note that we have to increase the threshold parameters slightly compared to those from Table \ref{tab:epsilon} for Algorithm \ref{alg:dft} in order to achieve similar rates of correct recovery, due to the instability caused by the computation of $\widehat{\y}$ from the given noisy data $\z^{\widehat{\mathrm{II}}}$. For both support lengths Algorithm \ref{alg:dct} reconstructs the original vector better than \texttt{idct} for all considered noise levels. Further, we also investigate the rate of correct recovery of the support block; these results can be found in Table \ref{tab:error_dct}. 
\begin{table}[!ht]
\begin{center}
\begin{tabular}{ccccc}\toprule
 & \multicolumn{4}{c}{Rate of Correct Recovery in \% Using Algorithm \ref{alg:dct} for} \\ 
 \multirow{2}{*}{SNR}	& $m=100$ & $m=100$ & $m=1000$ & $m=1000$ \\
 & & $m'\leq3m=300$ & & $m'\leq3m=3000$ \\ \midrule
 0	& 48	& 45	& 48	& 41	\\
10	& 57	& 57	& 67	& 66	\\
20	& 79	& 79	& 81	& 81 	\\
30	& 93	& 93	& 94	& 94 	\\
40	& 96	& 96	& 97	& 97	\\
50	& 99	& 99	& 99	& 99	\\ \bottomrule 
\end{tabular}
\caption{Rate of correct recovery of the support of $\x$ in percent for Algorithm \ref{alg:dct}, without bounding $m'$ and with $m'\leq3m$, for the 100 random input vectors with support length $m=100$ and $m=1000$ from Figures \ref{fig:error_dct_100} and \ref{fig:error_dct_1000}.}
\label{tab:error_dct}
\end{center}
\end{table}
In the second and fourth column we consider $\x$ to be correctly reconstructed by $\x'$ if the support of the original vector $\x$ is contained in the $m'$-length support of $\x'$ found by Algorithm  \ref{alg:dct}, and in the third and fifth column we additionally require  that $m'\leq 3m$. Algorithm \ref{alg:dct} has a very high rate of correct recovery for SNR values of 20 and greater, and, at least in our experiments for those noise levels, the support length of $\x'$ is always at most $3m$.
\section*{Acknowledgement}
The authors gratefully acknowledge partial support for this work by the DFG in the framework of the GRK 2088.

{\small
\bibliographystyle{abbrv}
\bibliography{library.bib}


\end{document}